\documentclass[reqno,12pt]{amsart}

\usepackage{}
 \usepackage{mathrsfs}
  \usepackage{tensor}
\usepackage{amsfonts, amssymb,amsmath, amsthm, amsxtra, latexsym, amscd}
\usepackage[latin1]{inputenc}
\usepackage{graphicx}
\usepackage{bm}

\theoremstyle{definition}
\newtheorem*{thm*}{Theorem}

\newtheorem{thm}{Theorem}[section]
\newtheorem{cor}[thm]{Corollary}
\newtheorem{lem}[thm]{Lemma}
\newtheorem{prop}[thm]{Proposition}

\newtheorem{rem}[thm]{Remark}

\numberwithin{equation}{section}

\bigskip

\def\refn#1.#2{\expandafter\def\csname#1\endcsname{[#2]}}
\def\refnr#1.{\csname#1\endcsname}

\begin{document}

\baselineskip  1.2pc

\title[The  $L^p$-$L^q$ problems  of Bergman-type operators ]
{The  $L^p$-$L^q$ problems  of Bergman-type operators }
\author{Lijia Ding }
\address{School of Mathematical Sciences,
Peking University, 
Beijing, 100086,
P. R. China}
\address{School of Mathematical Sciences,
Fudan University, Shanghai, 200433, P. R. China}
\email{ljding@pku.edu.cn}

\author{Kai Wang }
\address{ School of Mathematical Sciences,
Fudan University, Shanghai, 200433, P. R. China}
\email{kwang@fudan.edu.cn}
\thanks{ The first author was partially supported by Fudan University Exchange Program (2018017). The second author was  partially supported by NSFC (11722102), the Alexander von Humboldt Foundation (1151823), Shanghai Pujiang Program (16PJ1400600).
}
\begin{abstract} Let $\mathbb{B}^d$ be the unit ball on the complex space $\mathbb{C}^d$ with normalized Lebesgue measure $dv.$ For $\alpha\in\mathbb{R},$  denote $k_\alpha(z,w)=\frac{1}{(1-\langle z,w\rangle)^\alpha},$  the Bergman-type integral operator $K_\alpha$ on $L^1(\mathbb{B}^d,dv)$ is defined by 
$$ K_\alpha f(z)=\int _{\mathbb{B}^d}k_\alpha(z,w)f(w)dv(w).$$ It is an important class of operators in the holomorphic function space theory over the unit ball. We also consider the integral operator $K_\alpha^+$ on $L^1(\mathbb{B}^d,dv)$ which is given by $$ K_\alpha^+ f(z)=\int _{\mathbb{B}^d}\vert k_\alpha(z,w)\vert f(w)dv(w).$$ In this paper, we completely  characterize the $L^p$-$L^q$ boundedness of $K_\alpha,K_\alpha^+$ and $L^p$-$L^q$ compactness of $K_\alpha.$  The results of boundedness are in fact the Hardy-Littlewood-Sobolev theorem but also prove the conjecture of  \cite{CF}  in the case of bounded domain $\mathbb{B}^d.$ Meanwhile,  a trace formula and some sharp norm estimates of $K_\alpha,K_\alpha^+$ are given.

\end{abstract}
\maketitle

\section{Introduction}
Let $\mathbb{B}^d$ be the unit ball on the complex space $\mathbb{C}^d$ with the normalized Lebesgue measure $dv.$ For $\alpha\in\mathbb{R},$ denote $\alpha$-order Bergman-type kernel function  $k_\alpha(z,w)$ on $\mathbb{B}^d\times\mathbb{B}^d$ by $$k_\alpha(z,w)=\frac{1}{(1-\langle z,w\rangle)^\alpha}.$$ 
 Clearly the $(d+1)$-order Bergman-type kernel function $k_{d+1}(z,w)$ is the standard Bergman kernel  on $\mathbb{B}^d.$    Denote  Bergman-type integral operator $K_\alpha$ on $L^1(\mathbb{B}^d,dv)$ by 
$$ K_\alpha f(z)=\int _{\mathbb{B}^d}k_\alpha(z,w)f(w)dv(w).$$
Such operators $ K_\alpha $ play an important role in complex analysis of several variables and operator theory; in particular, when $\alpha=d+1,$ $K_{d+1}$ is the standard Bergman projection over the unit ball $\mathbb{B}^d.$
Indeed, for any $\alpha>0,$ if restrict $K_\alpha$ to the holomorphic function space $H(\mathbb{B}^d),$ then every $K_\alpha $ is a spacial form of fractional radial differential operator $R^{s,t},$ which is a kind of very useful operators in the Bergman space theory on the unit ball, see Lemma \ref{kr}; many key results on Bergman spaces can be deduced from the fractional radial differential operators, see example for \cite{ZZ,Zhu}. On the other hand, the operators  $K_\alpha$ play a significant role in the characterization of weighted Bloch spaces and Lipschitz spaces over the  unit ball $\mathbb{B}^d,$ see  \cite{ZZ,Zhu,Zhu1}.
We also consider the kernel integral operator $K_\alpha^+$ on $L^1(\mathbb{B}^d,dv),$  which is given by 
$$ K_\alpha^+ f(z)=\int _{\mathbb{B}^d}\frac{f(w)}{|1-\langle z,w \rangle|^\alpha}dv(w).$$
The operators $ K_\alpha^+$ can be regarded as Riesz potential  operators over the bounded domain $\mathbb{B}^d.$ Comparing to the classical Riesz potential  operators  over real Euclidian space $\mathbb{R}^d,$ whose  basic result
concerning mapping properties  is the Hardy-Littlewood-Sobolev theorem, see \cite{Li,NR,Pl,SW} and references therein. For convenience, we write $L^p(\mathbb{B}^d,dv)$ in the simple form $L^p(\mathbb{B}^d)$ or  $L^p$  for any $1\leq p\leq\infty$ without confusion arises. In the present paper,
we mainly concern the $L^p$-$L^q$ problem for $K_\alpha$ and $K_\alpha^+,$ namely we consider the boundedness and compactness of $K_\alpha$ and $K_\alpha^+,$ 
$$K_\alpha,K_\alpha^+:L^p\rightarrow L^q,$$ for $1\leq p,q\leq \infty.$ Indeed, the results of $L^p$-$L^q$ boundedness   are  the Hardy-Littlewood-Sobolev theorem with respect to  $K_\alpha^+$ over the unit ball $\mathbb{B}^d.$

 Actually, on  more general bounded domain $\Omega$ with the normalized Lebesgue measure $dv$ in $\mathbb{C}^d,$ the $L^p$-$L^q$ boundedness for  Bergman-type operators and in particular  $L^p$-$L^p$  boundedness for the   standard Bergman projection had  attracted much interest in the past decades; the target spaces are even  Bloch spaces, Lipschitz spaces and Sobolev spaces \cite{KV,PS}. As we all know, it is trivial that  the standard Bergman projection $P$ is bounded for any bounded domain when $p=q=2.$ However, the problem becomes very complicated for general $1\leq p,q\leq \infty.$ Nevertheless, the known results show that   depends strongly on the property of  the domain $\Omega.$  When $\Omega$ is a strongly pseudoconvex domain with  sufficiently smooth boundary, then the standard  Bergman projection $P: L^p(\Omega) \rightarrow L^p(\Omega)$ is bounded for any $1<p<\infty,$ the conclusion is also true for the Bergman-type integral operators with the order of the kernel function no more than $d+1;$ we refer the reader to  \cite{FR,MS,PS} along this line. Indeed, in the case of unit ball, the boundedness of more general Bergman type integral operators were considered in \cite{FR,Zha,ZZ,Zhu,Zhu1}. However, if  $\Omega$ is a  bounded symmetric domain of tube type  with rank $\geq 2$, the boundedness of the standard  Bergman projection $P: L^p(\Omega) \rightarrow L^p(\Omega)$ is conjectured by M. Stein that $P: L^p(\Omega) \rightarrow L^p(\Omega)$ is only bounded when $p$ belongs to a finite interval around $p = 2;$  we refer the reader to \cite{BB,BGN} and references therein. Although the $L^p$-$L^q$  boundedness for standard  Bergman projection $P$ over  tube type domains with rank $\geq 2$ has been considered a long time,  it is still an open problem.

Now return to our unit ball setting.  In \cite{FW}, X. Fang and Z. Wang established a relation between the boundedness of  standard  Bergman projection and Berezin transform on the weighted Bergman spaces over the unit disc  $\mathbb{D}=\mathbb{B}^1.$   The compactness of standard  Bergman projection $K_{2}:L^\infty(\mathbb{D}) \rightarrow L^q(\mathbb{D})$ for $1\leq q<\infty$ was observed by K. Zhu in Section 3.6 of \cite{Zhu1}. Recently,  X. Fang and  G. Cheng et al \cite{CF} completely solved the  $L^p$-$L^q$  boundedness problem of $K_\alpha$ over the unit disc $\mathbb{D};$  they also considered the   $L^p$-$L^q$  boundedness of Bergman-type operator over the upper half plane $\mathbb{U}=\{z\in \mathbb{C}: \text{Im} (z)>0\}.$ Not long afterward,  G. Cheng et al \cite{CH} solved  the  $L^p$-$L^q$  boundedness problem of $K_\alpha$ in the spacial case $\alpha=1$ over the unit ball $\mathbb{B}^d$ for general $d\geq1.$ The main difficulty in the case of high dimensional  ball $\mathbb{B}^d$  is how to determine  the critical exponents $d+1$ and $d+2,$ see the following theorems. In the present paper we completely describe the $L^p$-$L^q$ boundedness of $K_\alpha,K_\alpha^+$ but also the $L^p$-$L^q$ compactness of $K_\alpha$ over the unit ball $\mathbb{B}^d(d\geq1).$ The results of boundedness for $K_\alpha$ completely  prove the  conjecture of  \cite{CF} but also  extend some classical results \cite{DH,Mc,PS,Zha,Zhu,Zhu1} in the case unit ball, the results of boundedness for $K_\alpha^+$ are essentially the Hardy-Littlewood theorem as mentioned before; however the results of compactness are almost entirely new.  Firstly, it is trivial  that $ K_\alpha,K_\alpha^+:L^p\rightarrow L^q$ are compact  for any $1\leq p,q\leq \infty$  when $\alpha\leq0.$ Thus we  only  concern the case $\alpha>0.$ The following five theorems are our main results.

{\noindent\bf Theorem 1\label{thm1}.} If $ d+1<\alpha<d+2,$ then the following conditions are equivalent:
\begin{enumerate}
\item $K_\alpha:L^p\rightarrow L^q$ is bounded;
\item $K_{\alpha}^+:L^p\rightarrow L^q$ is bounded;
\item $K_\alpha:L^p\rightarrow L^q$ is compact;  
\item $p,q$ satisfy one of the following inequalities:\begin{enumerate} 
\item[(a)]  $\frac{1}{d+2-\alpha}<p<\infty, \frac{1}{q}>\frac{1}{p}+\alpha-(d+1);$
\item [(b)] $p=\infty, q<\frac{1}{\alpha-(d+1)}.$
\end{enumerate}
\end{enumerate}

As an consequence of Theorem 1, the following Hardy-Littlewood-Sobolev inequality (HLS) is established over the bounded domain $\mathbb{B}^d.$

{\noindent\bf HLS 1\label{HLS1}.}  For any $1<p,s<\infty, \frac{1}{s}+\frac{1}{p}+\alpha< d+2$ and $d+1<\alpha< d+2,$ then there exists  a constant $C$ which  depends only on $p,\alpha,d,s$ such that
\begin{equation}\label{hls}\left \vert\int_{\mathbb{B}^d}\int_{\mathbb{B}^d}\frac{f(w)g(z)}{|1-\langle z,w \rangle|^\alpha}dv(w)dv(z)\right\vert\leq C\Vert f\Vert_{L^p}\Vert g\Vert_{L^s},\end{equation}
for all $f\in L^p(\mathbb{B}^d),g\in L^s(\mathbb{B}^d).$

{\noindent\bf Theorem 2\label{th}.} If $ 0 <\alpha\leq d+1,$ then the following conditions are equivalent:
\begin{enumerate}
\item $K_\alpha:L^p\rightarrow L^q$ is bounded;  
\item $K_\alpha^+:L^p\rightarrow L^q$ is bounded;
\item $p,q$ satisfy one of the following inequalities:\begin{enumerate}
\item[(a)] $ p=1, q<\frac{d+1}{\alpha};$
\item[(b)] $1<p<\frac{d+1}{d+1-\alpha}, \frac{1}{q}\geq \frac{1}{p}+ \frac{\alpha}{d+1}-1;$ 
\item[(c)] $p=\frac{d+1}{d+1-\alpha}, q < \infty;$
\item[(d)] $\frac{d+1}{d+1-\alpha}<p\leq\infty.$
\end{enumerate}
\end{enumerate}

In particular, $K_\alpha,K_\alpha^+:L^p\rightarrow L^p$ are both bounded for any $1\leq p\leq \infty$ when $ 0 <\alpha< d+1,$ which is actually  a more precise conclusion than Lemma 5 of  \cite{PS} in the case of unit ball. Although $K_\alpha,K_\alpha^+:L^1\rightarrow L^\frac{d+1}{\alpha}$ are both unbounded under the condition of Theorem 2, it turns out  that $K_\alpha$ is weak type $(1,\frac{d+1}{\alpha}),$ i.e. $K_\alpha,K_\alpha^+: L^1\rightarrow L^{\frac{d+1}{\alpha},\infty}$ are both bounded over $\mathbb{B}^d,$ see the following Corollary \ref{wea}, which is a generalization of the result that the standard  Bergman projection is weak type (1,1) over some bounded domains \cite{DH,Mc}. More importantly, by Theorem 2, it implies the following the Hardy-Littlewood-Sobolev inequality over the unit ball $\mathbb{B}^d.$ 

{\noindent\bf HLS 2\label{HLS2}.}
For any $1<p,s<\infty, \frac{1}{s}+\frac{1}{p}+\frac{\alpha}{d+1}\leq2$ and $\alpha\leq d+1,$ then there exists  a constant $C$ that  depends only on $p,\alpha,d,s$ satisfying that (\ref{hls}) holds 
for all $f\in L^p(\mathbb{B}^d),g\in L^s(\mathbb{B}^d).$ 

Comparing HLS 1 and HLS 2 to the classical  Hardy-Littlewood-Sobolev inequality \cite{Li,NR,Pl,SW} over $\mathbb{R}^d,$ it is surprising  that HLS 1 is a new type of  Hardy-Littlewood-Sobolev inequality.

{\noindent\bf Theorem 3\label{th3}.} If $ 0 <\alpha\leq d+1,$ then the following conditions are equivalent:
\begin{enumerate}
\item $K_\alpha:L^p\rightarrow L^q$ is compact;  
\item $p,q$ satisfy one of the following inequalities:\begin{enumerate}
\item[(a)] $ p=1, q<\frac{d+1}{\alpha};$
\item[(b)] $1<p<\frac{d+1}{d+1-\alpha}, \frac{1}{q}> \frac{1}{p}+ \frac{\alpha}{d+1}-1;$ 
\item[(c)] $p=\frac{d+1}{d+1-\alpha}, q < \infty;$
\item[(d)] $\frac{d+1}{d+1-\alpha}<p\leq\infty.$
\end{enumerate}
\end{enumerate}

{\noindent\bf Theorem 4.} For  $\alpha\in\mathbb{R},$ then the following conditions are equivalent:
\begin{enumerate}
\item $\alpha< d+2;$
\item there exist $1\leq p,q\leq \infty $ such that $K_\alpha:L^p\rightarrow L^q$ is bounded;
\item there exist $1\leq p,q\leq \infty $ such that $K_\alpha^+:L^p\rightarrow L^q$ is bounded;
\item there exist $1\leq p,q\leq \infty $ such that $K_\alpha:L^p\rightarrow L^q$ is compact.
\end{enumerate}
 
{\noindent\bf Theorem 5\label{thm4}.} If $\alpha<\frac{d+2}{2},$ then the following holds.
\begin{enumerate}
\item  $K_\alpha,K_\alpha^+:L^2\rightarrow L^2$ are Hilbert-Schmidt.
\item Moreover, if $d=1$ and $0<\alpha<\frac{3}{2},$ then we have the trace formula, $$Tr(K_\alpha^*K_\alpha)=\Vert K_{2\alpha}^+\Vert_{L^\infty\rightarrow L^1}=\frac{1}{(\alpha-1)^2}\left (\frac{\Gamma(3-2\alpha)}{\Gamma^2(2-\alpha)} -1\right ).$$ where $\Gamma$  is the usual Gamma function.  When $\alpha=1,$ the  quantity on the right side   should be interpreted as $\frac{\pi^2}{6}.$
\end{enumerate}

The above theorems show that $K_\alpha:L^p\rightarrow L^q$ is bounded if and only if $K_\alpha^+:L^p\rightarrow L^q$ is bounded. From Theorem 1, it is  amazing to know that, when  $d+1<\alpha<d+2,$  $K_\alpha:L^p\rightarrow L^q$ is compact if and only if $K_\alpha:L^p\rightarrow L^q$ is bounded. However,  it is very different when $0<\alpha\leq d+1$ by  Theorems 2 and 3. In particular,    the  standard  Bergman projection $K_{d+1}:L^p\rightarrow L^q$ is compact if and only if  $1\leq q<p\leq\infty$ over $\mathbb{B}^{d}.$

 Let us consider the above boundedness problem in the following viewpoint. Denote $G(K_\alpha)$ by the set of $(\frac{1}{p},\frac{1}{q})\in E$ such that $K_\alpha: L^p\rightarrow L^q$ is bounded, where $E$ is given by $$ E=\{(x,y)\in \mathbb{R}^2:  0\leq x,y\leq 1 \},$$ i.e. $E$ is a unit square in the real  plane $\mathbb{R}^2.$ Following by T. Tao \cite{Tao},  $G(K_\alpha)$ is called the type diagram of the operator $K_\alpha,$ see Figure \ref{fi}. By a classical  interpolation result, it implies immediately that every $G(K_\alpha)$ is convex. The adjointness of $K_\alpha$ implies that   $G(K_\alpha)$ is axisymmetric on the inside of $E$. To prove the above theorems is equivalent to solve the  corresponding type diagrams. The above theorems show that the type diagram  $G(K_\alpha)$ is determined by the corresponding inequalities. Conversely, the  inequalities in the above theorem are determined by the type diagram  $G(K_\alpha).$ The convexity and axisymmetry of the type diagram will make the solving process simpler.  Similarly, we can define the type diagrams $G(K_{\alpha}^+)$ for operators $K_{\alpha}^+,$ which are also convex and is axisymmetric on the inside of $E$; see Figure \ref{fi}. Note that $|K_\alpha(f)|\leq K_\alpha^+(|f|),$ it  implies  immediately  that $G(K_{\alpha}^+)\subset G(K_{\alpha}).$  Then combing with several embedding theorems  of holomorphic function spaces and some estimations of Bergman kernel over the unit ball, we completely  characterize $L^p$-$L^q$ boundedness and  $L^p$-$L^q$ compactness of $K_\alpha.$ The above main theorems show in fact that $G(K_{\alpha}^+)= G(K_{\alpha})$ for every $\alpha\in \mathbb{R}.$ After characterizing the boundedness and compactness of $K_\alpha,$ by using of the hypergeometric function theory and  the interpolation theory, we give some sharp norm estimations of $K_\alpha,K_\alpha^+.$ It is in fact that we estimate the upper bounds of the best constant in the inequalities HLS 1 and HLS 2.

  The results of this paper can be generalized to cover the weighted  Lesbegue integrable spaces and more general kernel operators over the unit ball.  Another promising idea is the study of the boundedness of Bergman projection over the  bounded symmetric domains of tube type \cite{BGN}  with rank $\geq 2.$
  
 The paper is organized as follows. In Section 2, we give some basic properties of the operators $K_\alpha.$ In Section 3, we prove  Theorem 1.  The proof of Theorem 2 is given in Section 4. In Section 5, we prove Theorem 3 and Theorem 4. Finally, we give some sharp norm estimations of the operators $K_\alpha,K_\alpha^+.$
 \begin{figure}\label{fi}
  \centering
  \includegraphics[width=.8\textwidth]{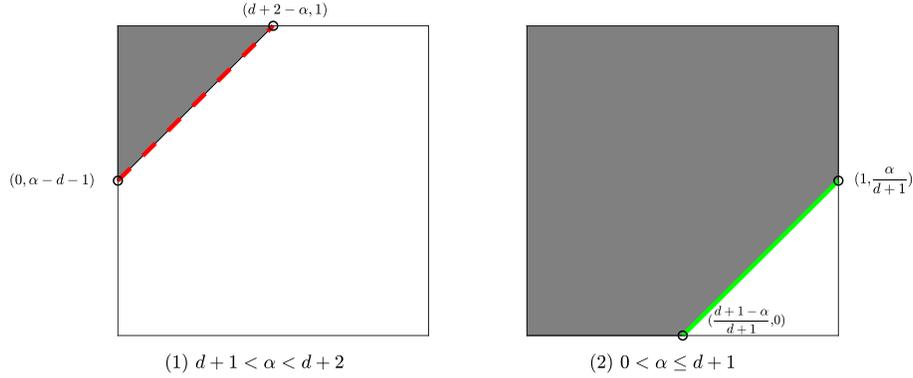} 
  \caption{Type diagrams $G(K_\alpha),G(K_\alpha^+).$}
  \label{img} \end{figure}

\section{Basic properties of $K_\alpha$}
In this section, we prove some results for latter use. We first take a rough look at the property of type diagram $G(K_\alpha)$ of the operator $K_\alpha.$ We prove that every $G(K_\alpha)$ is convex and  is axisymmetric on the inside of $E$ as mentioned before.  Let $l_E$ be the diagonal line of the  square $E$ which connects  points $(0,1)$ and $(1,0).$ Clearly $G(K_\alpha)\subset E$ for any $\alpha\in \mathbb{R}.$
\begin{prop}\label{cs}\begin{enumerate}
\item If $G(K_\alpha)\neq \emptyset,$ then $(0,1)\in G(K_\alpha)$; if $(1,0)\in G(K_\alpha),$ then $ G(K_\alpha)=E.$
\item For any $\alpha\in \mathbb{R},$ the type diagram $G(K_\alpha)$ is  convex and  is axisymmetric about $l_E$ on the inside of $E.$ \end{enumerate}\end{prop}
\begin{proof} 
{\noindent (1)} It comes from the following continuous embedding of $L$-integrable spaces, i.e. $L^p\subset L^q$ whenever $p\geq q.$

{\noindent (2)} To show that $G(K_\alpha)$ is  convex, it suffices to show that if $(\frac{1}{p_1},\frac{1}{q_1}),(\frac{1}{p_2},\frac{1}{q_2}) \in G(K_\alpha),$ then $\theta(\frac{1}{p_1},\frac{1}{q_1})+(1-\theta)(\frac{1}{p_2},\frac{1}{q_2}) \in G(K_\alpha)$ for any $0\leq\theta\leq1.$ Indeed, it is a direct corollary of the following Lemma \ref{int}, a classical complex interpolation result.  Now we turn to the symmetry. By Fubini's theorem, it implies that $K_{\alpha}$ is adjoint. Then, for $1<p,q<\infty,$ the boundedness of $K_{\alpha}: L^p\rightarrow L^q$ is equivalent to  the boundedness of $K_{\alpha}: L^{q'}\rightarrow L^{p'},$ where $p',q'$ are the conjugate numbers of $p,q$, respectively.
    It means that $(\frac{1}{p},\frac{1}{q}) \in G(K_\alpha)$ if and only if $(\frac{1}{q'},\frac{1}{p'})\in G(K_\alpha).$ It easy to check that $(\frac{1}{p},\frac{1}{q}) $ and $(\frac{1}{q'},\frac{1}{p'})$ are  symmetric about $l_E$ by the conjugate relationship.  \end{proof}
    \begin{lem}\label{int}\cite{Zhu} Suppose $1\leq p_1,p_2,q_1,q_2\leq\infty.$  If a linear operator $T$ such that $T :L^{p_1}\rightarrow L^{q_1}$ is bounded with norm $M_1$ and $T :L^{p_2}\rightarrow L^{q_2}$ is bounded with norm $M_2.$ Then $T :L^{p}\rightarrow L^{q}$ is bounded with norm no more than $M_1^{\theta}M_{2}^{1-\theta},$  if   there exists  $\theta\in (0,1)$ such that  $$\frac{1}{p}=\frac{\theta}{p_1}+\frac{1-\theta}{p_2}, \frac{1}{q}=\frac{\theta}{q_1}+\frac{1-\theta}{q_2}.$$  
\end{lem}
\begin{rem} Proposition \ref{cs} shows that the type diagram $G(K_\alpha)$ is a bounded convex set in the plane $\mathbb{R}^2$, so to solve $G(K_\alpha),$ it suffices to find out all  extreme points  or the boundary points of $G(K_\alpha).$ The symmetry of $G(K_\alpha)$ shows that is only need to find out a half. On the other hand, Proposition \ref{cs} holds for more general domains and adjoint operators.
\end{rem}
\begin{cor}\label{cs1}\begin{enumerate}
\item If $G(K_\alpha^+)\neq \emptyset,$ then $(0,1)\in G(K_\alpha^+)$; if $(1,0)\in G(K_\alpha^+),$ then $ G(K_\alpha^+)=E.$
\item For any $\alpha\in \mathbb{R},$ the type diagram $G(K_\alpha^+)$ is  convex and  is axisymmetric about $l_E $ on the inside of $E.$ \end{enumerate}\end{cor}
\begin{cor}\label{gke} If $\alpha\leq0,$ then $G(K_\alpha)=G(K_\alpha^+)=E.$  
\end{cor}
Corollary \ref{gke} means that $K_{\alpha},K_{\alpha}^+: L^p\rightarrow L^q$ are bounded for any $1\leq p,q\leq\infty$ if $\alpha\leq 0.$
For any $\beta>-1,$ denote $dv_{\beta}(z)=c_\beta(1-|z|^2)^{\beta}dv(z),$ where $c_\beta=\frac{\Gamma(d+\beta+1)}{\Gamma(d+1)\Gamma(\beta+1)}.$ For $1\leq p\leq \infty,$ let $A_\beta^p=H(\mathbb{B}^d)\cap L^p(dv_{\beta})$ be the weighted Bergman space on $\mathbb{B}^d,$ in particular,  $A_\beta^\infty=H^\infty$ is just the bounded holomorphic function space. Recall that  $K_{d+1}$ is the Bergman projection from $L^p$ onto $A_0^p,$ a well known result is that $K_{d+1}(L^p)=A_0^p$ for $1<p<\infty.$ Now we establish a general result for $\alpha\geq d+1.$
\begin{prop}\label{kka} Suppose that  $\alpha\geq d+1$ and  $1<p<\infty,$ then $$K_{\alpha}(L^p)=K_{\alpha}(A_0^p)=A_{p(\alpha-d-1)}^p.$$
\end{prop}
To prove Proposition \ref{kka}, we need some lemmas. The following Lemma \ref{3k} was proved \cite{CF} in the case $d=1,$ use the same method, it can be proved in the general case,  see Lemma 11 of  \cite{CF}  for  more detail.
\begin{lem} \label{3k} If $\alpha>0$ and $1<p<\infty,$ then 
$$K_{\alpha}K_{d+1}=K_{\alpha} ~\text{on} ~L^p.$$
\end{lem}
Lemma \ref{3k} shows that for  $1<p,q<\infty$,  $K_\alpha: L^p\rightarrow L^q$ is bounded if and only if $K_\alpha: A_0^p\rightarrow A_0^q$ is bounded. Now we turn to the behavior of $K_\alpha$ on holomorphic function spaces. Recall first the definition of fractional radial differential operator $R^{s,t}$ on $H(\mathbb{B}^d).$

For any two real parameters $s$ and $t$ with the property that neither $d + s $ nor $d + s + t$ is a negative integer, the invertible operator $R^{s,t}$ is given by 
  $$R^{s,t}f(z)=\sum_{n=0}^{\infty}\frac{\Gamma(d+1+s)\Gamma(d+1+n+s+t)}{\Gamma(d+1+s+t)\Gamma(d+1+n+s)}f_n(z),$$ for any $f=\sum_{n=0}^{\infty}f_n\in H(\mathbb{B}^d)$  with homogeneous expansion. In fact, it can be checked by direct calculation that the invertible operator of  $R^{s,t}$ is just $R^{s+t,-t}.$ Be careful of  the  invertible operator here merely means that is linear.
 \begin{lem}\label{kr} For $\alpha>0$ and $1<p<\infty,$ the following holds on $A_0^p,$ $$K_\alpha=R^{0,\alpha-d-1}. $$
 \end{lem}
  \begin{proof} Suppose $f=\sum_{n=0}^{\infty}f_n\in A_0^p$ with the homogeneous expansion. By direct calculation, it implies that 
    \begin{equation}\label{KR}K_\alpha f=\sum_{n=0}^{\infty} \frac{\Gamma(d+1)\Gamma(\alpha+n)}{\Gamma(\alpha)\Gamma(d+1+n)}f_n.\end{equation}
    It leads to the desired result.
\end{proof}
 {\noindent{\bf{Proof of  Proposition \ref{kka}.}} Lemma \ref{3k} implies that $K_{\alpha}(L^p)=K_{\alpha}(A_0^p).$ Now we prove $K_{\alpha}(A_0^p)=A_{p(\alpha-d-1)}^p.$  By Theorem 14 of \cite{ZZ}, which is a characterization of Bergman space, shows that $f\in A_0^p$ if and only if $R^{0,\alpha-d-1}f\in L^p(dv_{p(\alpha-d-1)}),$ namely $f\in A_0^p$ if and only if $R^{0,\alpha-d-1}f\in A_{p(\alpha-d-1)}^p.$
 Note that $K_\alpha=R^{0,\alpha-d-1}$ by Lemma \ref{kr}, it follows that $f\in A_0^p$ if and only if $K_\alpha f\in A_{p(\alpha-d-1)}^p.$ It shows that $K_\alpha (A_0^p)\subset A_{p(\alpha-d-1)}^p.$ To prove another direction, suppose that $g\in A_{p(\alpha-d-1)}^p.$ Since $K_\alpha=R^{0,\alpha-d-1}$ is  invertible on $H(\mathbb{B}^d),$ i.e. there exists $f\in H(\mathbb{B}^d)$ such that $K_\alpha f=R^{0,\alpha-d-1}f=g.$ From Theorem 2.19 of \cite{Zhu}, there exists a positive constant $c$ that depends  only on $\alpha,d,p$ such that  
 $$ \Vert f\Vert_{L^p}\leq c \Vert g \Vert _{A_{p(\alpha-d-1)}^p}.$$ It means that $f\in A_0^p.$ Thus $A_{p(\alpha-d-1)}^p\subset K_\alpha (A_0^p).$ It completes the proof. \qed
 \begin{cor}\label{kar} Suppose that  $\alpha\geq d+1$ and  $1<p<\infty,$ then for any $\gamma>-1,$ the following holds,$$K_\alpha(L^p(dv_\gamma))=K_\alpha(A_\gamma^p)=A_{\gamma+p(\alpha-d-1)}^p.$$ \end{cor}
 The following Proposition \ref{kaa} gives the image of $K_\alpha$ in case of $p=\infty.$  Denote   $\mathcal{B}_{\beta}$ by the  weighted Bloch space   on $\mathbb{B}^d,$  see  definition for Section 7.1 of \cite{Zhu}.

 \begin{prop}\label{kaa} For $\alpha\geq d+1,$ then 
 $K_{\alpha}(H^\infty)\subsetneq K_{\alpha}(L^\infty)=\mathcal{B}_{\alpha-d}.$ 
 \end{prop}
  \begin{proof} 
  Note that $K_{\alpha}(L^\infty)=\mathcal{B}_{\alpha-d}$ by Theorem 7.1 of \cite{Zhu}. If $\alpha=d+1,$ then $K_{d+1}(H^\infty)=H^\infty,$ thus $K_{d+1}(H^\infty)\subsetneq \mathcal{B}_{\alpha-d}.$ Now turn to the case $\alpha>d+1.$ Note that $K_{\alpha}(H^\infty)\subset K_{\alpha}(A_{0}^p)$ for any $1<p<\infty,$ then it implies by Proposition \ref{kka} that \begin{equation}\label{kh}K_{\alpha}(H^\infty)\subset\bigcap_{1<p<\infty} A_{p(\alpha-d-1)}^p.\end{equation}
  On the other hand, from Theorem 2.1 of \cite{Zhu}, a pointwise estimates for functions in weighted Bergman spaces, we know that \begin{equation}\label{apb}A_{\gamma}^p\subset \mathcal{B}_{\frac{d+1+\gamma}{p}}.\end{equation}  
  Combing (\ref{kh}) with (\ref{apb}), it implies that $$K_{\alpha}(H^\infty)\subset\bigcap_{1<p<\infty}\mathcal{B}_{(\alpha-d)+\frac{d+1}{p}-1}.$$
  Together with the fact that the weighted Bloch space is strictly increased, namely $\mathcal{B}_\beta\subsetneq \mathcal{B}_{\beta'}$ whenever $0<\beta<\beta',$
it implies that $K_{\alpha}(H^\infty)\subsetneq \mathcal{B}_{\alpha-d}.$
  \end{proof}
  \begin{rem} The  monotonicity of the weighted Bloch space can be obtained as follows. It is easy to see that the weighted Bloch space is increased, so it suffices to show that is strict. For any $0<\beta<\beta',$ there exist $p>1$ and $\varepsilon>0$ such that $$\beta<\beta-1+\frac{d+\varepsilon}{p}<\beta'.$$ Combing (\ref{apb}) and the following Lemma \ref{ba}, it implies that $$\mathcal{B}_\beta\subsetneq A_{p(\beta-1)-1+\varepsilon}^p\subset \mathcal{B}_{\beta'}.$$
  \end{rem}
\section{Proof of Theorem \ref{thm1}} 
In this section, we prove Theorem \ref{thm1}. We need several embedding theorems of holomorphic function spaces on the unit ball $\mathbb{B}^d.$ For convenience, we state them without proof as follows. 
\begin{lem}\cite{ZZ}\label{aa} Let $0<q<p<\infty$.Then $A_\beta^p\subset A_\gamma^q $ if and only if $\frac{\beta+1}{p}<\frac{\gamma+1}{q}.$ And in this case the inclusions are strict.
\end{lem}
\begin{proof} See proof of  Theorem 70 of \cite{ZZ}. \end{proof}
\begin{lem}\cite{OC,ZZ} \label{ba} Suppose that $\beta>0,\gamma>-1,p\geq1,$  then $\mathcal{B}_{\beta} \subset A_\gamma^p$ if and only if   $\beta-1<\frac{1+\gamma}{p}.$ And in this case the inclusions are strict.
\end{lem}
\begin{proof} See proofs  in \cite{OC} or  Theorem 66 of \cite{ZZ}. \end{proof}

We also needs the following lemmas. 

\begin{lem}\label{kiq} If $  d+1<\alpha<d+2,  $ then $K_\alpha:L^{\infty}\rightarrow L^q$ is bounded if and only if $q<\frac{1}{\alpha-(d+1)}.$ 
\end{lem}
\begin{proof} We first  to show that $K_\alpha:L^{\infty}\rightarrow L^q$ is bounded if 
$q< \frac{1}{\alpha-(d+1)}.$ Then, for $f\in L^{\infty},$ by Proposition 1.4.10 of  \cite{Rud} and H\"older's inequality, it implies that  \begin{equation}\begin{split}\label{kf}
|K_{\alpha}f(z)|\leq \Vert f \Vert_{\infty}\int_{\mathbb{B}^d} \frac{1}{|1-\langle z,w \rangle|^\alpha}dv(w)\leq  C_{d,\alpha}  \Vert f\Vert_{\infty} (1-|z|^2)^{d+1-\alpha}, |z|\rightarrow 1^-,
\end{split}\end{equation} where $ C_{d,\alpha} $ is a constant.
The condition $q<\frac{1}{\alpha-(d+1)}$ means that $q((d+1)-\alpha )>-1.$ Then (\ref{kf}) implies that  $K_{\alpha}f(z)\in L^q$ and $K_\alpha:L^{\infty}\rightarrow L^q$ is bounded. Now we turn to prove that $K_\alpha:L^{\infty}\rightarrow L^q$ is unbounded if 
$q\geq \frac{1}{\alpha-(d+1)}.$ By H\"older's inequality, it is enough to prove that  $K_\alpha:L^{\infty}\rightarrow L^{\frac{1}{\alpha-(d+1)}}$ is unbounded. 
 It suffices to show that  $K_\alpha(L^{\infty})\not\subset L^{\frac{1}{\alpha-(d+1)}}.$ Since $K_\alpha(L^{\infty})=\mathcal{B}_{\alpha-d}$, it suffices to show that  $\mathcal{B}_{\alpha-d}\not\subset A_0^{\frac{1}{\alpha-(d+1)}}.$ Indeed, it is a fact from  Lemma \ref{ba}. \end{proof}
\begin{cor}\label{kp1} If $ d+1<\alpha<d+2,  $ then $K_\alpha:L^{p}\rightarrow L^1$ is bounded if and only if $p>\frac{1}{(d+2)-\alpha }.$ 
\end{cor}
\begin{proof} First, suppose that $p>\frac{1}{(d+2)-\alpha }.$ From Lemma \ref{kiq} and $K_\alpha$ is an adjoint operator, we know that $K_\alpha: L^p\rightarrow (L^{\infty})^*$ is bounded if $p>\frac{1}{(d+2)-\alpha }.$ Proposition \ref{kka} implies that $K_\alpha(L^{p})=A_{p(\alpha-d-1)}^p.$ Since $\frac{p(\alpha-d-1)+1}{p}<(\alpha-d-1)+(d+2)-\alpha=1,$ it follows by Lemma \ref{aa} that $A_{p(\alpha-d-1)}^p\subset A_0^1.$ Thus $K_\alpha(L^{p})\subset L^1.$ Note that $L^1\subset (L^{\infty})^*,$ it implies that $K_\alpha: L^p \rightarrow L^1$ is bounded.

{\noindent{}} Conversely, suppose that $K_\alpha:L^{p}\rightarrow L^1,p\neq \infty$ is bounded. Then $K_\alpha:L^\infty \rightarrow L^{p'}$ is bounded, where $p'=\frac{p}{p-1}.$ From Lemma \ref{kiq}, it implies that 
$\frac{p}{p-1}=p'<\frac{1}{\alpha-(d+1)},$ it means that $p>\frac{1}{(d+2)-\alpha }.$ Clearly the case of $p=\infty$ is trivial by Lemma \ref{kiq}.
\end{proof}
\begin{cor}\label{kl} If $ d+1<\alpha<d+2,$ then  \begin{enumerate}
\item  $K_\alpha ^+:L^{\infty}\rightarrow L^q$ is bounded if and only if $q<\frac{1}{\alpha-(d+1)};$ 
\item   $K_\alpha ^+:L^{p}\rightarrow L^1$ is bounded if and only if $p>\frac{1}{(d+2)-\alpha }.$ 
\end{enumerate}
\end{cor}
\begin{proof} (1)  For $f\in L^{\infty},$ by Proposition 1.4.10 of  \cite{Rud} and H\"older's inequality, it implies that  \begin{equation}\begin{split}\label{kff}
|K_{\alpha}^+f(z)|\leq \Vert f\Vert_{\infty}\int_{\mathbb{B}^d} \frac{1}{|1-\langle z,w \rangle|^\alpha}dv(w)\leq  C_{d,\alpha} \Vert f\Vert_{\infty} (1-|z|^2)^{d+1-\alpha}, |z|\rightarrow 1^-,
\end{split}\end{equation} where $C_{d,\alpha} $ is a constant.
So, if  $q<\frac{1}{\alpha-(d+1)},$ i.e.  $q((d+1)-\alpha )>-1,$ then (\ref{kff}) implies that  $K_{\alpha}f(z)\in L^q$ and $K_\alpha ^+:L^{\infty}\rightarrow L^q$ is bounded. It means that $$\{ (0,\frac{1}{q}):\frac{1}{q}>\alpha-(d+1)\}\subset G(K_{\alpha}^+).$$
On the other hand, Lemma \ref{kiq} implies that point $ (0,\frac{1}{q})\in G(K_{\alpha}) $ if and only if $\frac{1}{q}>\alpha-(d+1).$ Combing with $G(K_{\alpha}^+)\subset G(K_{\alpha}),$ it follows that 
$ (0,\frac{1}{q})\in G(K_{\alpha}) $ if and only if $\frac{1}{q}>\alpha-(d+1).$ It leads the desired result.

{\noindent{(2)}} The proof is similar to (1).
\end{proof}

\begin{lem}\label{pqu} Suppose  that $ d+1<\alpha<d+2$ and $\frac{1}{q}\leq\frac{1}{p}+\alpha-(d+1),$ then $K_\alpha:L^{p}\rightarrow L^q$ is unbounded. 
\end{lem} 
\begin{proof} By the continuous embedding of $L$-integrable spaces, it suffices to show that  $K_\alpha:L^{p}\rightarrow L^q$ is unbounded if  $ d+1<\alpha<d+2,\frac{1}{q}=\frac{1}{p}+\alpha-(d+1).$ The cases of $p=\infty $ or  $q=1$ had been proved in Lemma  \ref{kiq} and Corollary \ref{kp1}. For case   of $1<p,q<\infty,$ it suffices to show that $K_\alpha(L^p)\not\subset L^q.$  On the other hand, Proposition \ref{kka} shows that $K_\alpha(L^p)=A_{p(\alpha-d-1)}^p,$ a holomorphic function space.  Thus, it suffices to show that \begin{equation}\label{kppq} K_\alpha(L^p)=A_{p(\alpha-d-1)}^p\not\subset A_0^q.\end{equation}
         Since  $\frac{p(\alpha-d-1)+1}{p} =\frac{1}{q},$  it follows that  (\ref{kppq}) holds by Lemma \ref{aa}. It completes the proof.
\end{proof}
{\noindent{\bf{Proof of Theorem 1.}}

{\noindent{$Step~1$.} To prove that (1)$\Leftrightarrow$(2)$\Leftrightarrow$(4).

First, we prove that (1) is equivalent to (4).  As mentioned before, it is equivalent to prove that $G(K_\alpha)$ is exactly the triangle region $D_1\subset E$  which determined by the equations in (4) of Theorem 1, namely $G(K_\alpha)=D_1$. Lemma \ref{kiq}, Corollary \ref{kp1} and the convexity of $G(K_\alpha)$ imply that $ D_1\subset G(K_\alpha).$ On the other hand, Lemma \ref{pqu} and the convexity of $G(K_\alpha)$ imply that $E-D_1\subset E -G(K_\alpha),$ it follows that  $G(K_\alpha)\subset D_1.$ Thus $G(K_\alpha)=D_1.$ Now we turn to prove that (2) is equivalent to (4), it is equivalent to prove that
         $G(K_\alpha^+)=D_1.$ Corollary \ref{kl} and the convexity of $G(K_\alpha^+)$ implies that $D_1\subset G(K_\alpha^+).$ Combing the fact  that $G(K_\alpha^+)\subset G(K_\alpha)=D_1,$ then $G(K_\alpha^+)=D_1.$ It completes the proof. 
         
 {\noindent{$Step~2$.} To prove that (1)$\Leftrightarrow$(3).        
 
Since compact operators must be bounded, it suffices to prove that 
$$K_\alpha:L^p\rightarrow L^q ~\text{is~ compact},~\text{if } (\frac{1}{p},\frac{1}{q})\in G(K_\alpha).$$ 
We first prove the following claim.  

{\noindent{$Claim:$}} $K_\alpha:L^\infty\rightarrow L^q$ is compact if and only if $q<\frac{1}{\alpha-(d+1)}.$  

If  $K_\alpha:L^\infty\rightarrow L^q$ is compact, is immediate from  Corollary \ref{kp1} that $q<\frac{1}{\alpha-(d+1)}.$ Now we prove the reverse, that is, to prove that   $K_\alpha:L^\infty\rightarrow L^q$ is compact if $q<\frac{1}{\alpha-(d+1)}.$ We need to show that for any bounded sequence  in  $L^\infty$, there is a subsequence such that whose image under $K_\alpha$  converges in $L^q.$ Suppose that $\{ f_n\}\in L^\infty$ is an arbitrary bounded sequence and $K$ is an arbitrary compact subset of $\mathbb{B}^d.$ Moreover, we assume that $\Vert f_n\Vert_\infty \leq C$ for any $n\geq 1,$ where $C$ is a positive constant. 
Then we obtain \begin{equation}\begin{split} \label{} \sup_{z\in K} \vert K_\alpha f_n(z)\vert&\leq\Vert f_n\Vert_\infty \sup_{z\in K} \int _{\mathbb{B}^d}\frac{1}{\vert 1-\langle z,w\rangle\vert^{\alpha}}dv(w)\notag\\
&\leq \Vert f_n\Vert_\infty \sup_{z\in K}\frac{1}{(1-\vert z\vert)^{\alpha}}< \infty.\end{split}\end{equation}  Combing with that the image of $K_\alpha$ is holomorphic, it implies that  $\{ K_\alpha f_n\}$ is a normal family. Hence $\{ f_n\}$ has a 
subsequence $\{f_{n_j} \}$ such that $K_\alpha f_{n_j}$ converges uniformly on compact subsets of $\mathbb{B}^d$ to a holomorphic function $g$. By Fatou's Lemma and boundedness of $K_\alpha$, it follows that 
\begin{equation}\label{Fa}\int_{\mathbb{B}^d} \vert g\vert^q dv\leq \varliminf_{j\rightarrow\infty}\int_{\mathbb{B}^d} \vert K_\alpha f_{n_j}\vert ^qdv\leq \Vert K_\alpha\Vert_{L^\infty\rightarrow L^q}^q \varliminf_{j\rightarrow\infty}\Vert f_{n_j}\Vert_\infty ^q<\infty.\end{equation}
It means that $g\in L^q.$ Now we prove that there exists positive function  $g_1\in L^q $ such that $\vert K_\alpha f_{n_j}\vert \leq g_1.$ We first observe that  Proposition 1.4.10 of \cite{Rud} and the condition $q<\frac{1}{\alpha-(d+1)}$ imply that $$\left (\int _{\mathbb{B}^d}\frac{1}{\vert 1-\langle z,w\rangle\vert^{\alpha}}dv(w)\right)^{q} \in L^1.$$
Then by  easy estimation, it  yields that 
\begin{equation}\begin{split} \label{L1}\vert K_\alpha f_{n_j}(z)\vert &\leq\Vert f_{n_j}\Vert _\infty \int _{\mathbb{B}^d}\frac{1}{\vert 1-\langle z,w\rangle\vert^{\alpha}}dv(w)\\
&\leq C\int _{\mathbb{B}^d}\frac{1}{\vert 1-\langle z,w\rangle\vert^{\alpha}}dv(w).\\
 \end{split}\end{equation}
Thus (\ref{L1}) shows that it is enough to take $g_1=C\int _{\mathbb{B}^d}\frac{1}{\vert 1-\langle z,w\rangle\vert^{\alpha}}dv(w).$ Combing (\ref{Fa})  with (\ref{L1}), it implies that $$ \vert K_\alpha f_{n_j}-g\vert^q\leq( g_1+\vert g\vert)^q\in L^1,\forall j\geq1.$$ 
By dominated convergence theorem, it  gives that 
$$\lim_{ j\rightarrow\infty} \Vert K_\alpha f_{n_j}-g \Vert_q=\lim_{j\rightarrow\infty}\left (\int _{\mathbb{B}^d}\vert  K_\alpha f_{n_j}-g \vert^q dv\right)^\frac{1}{q}=\left (\int _{\mathbb{B}^d}\lim_{j\rightarrow\infty} \vert  K_\alpha f_{n_j}-g\vert^q dv\right)^\frac{1}{q}=0,$$ and completes the proof the claim.

{\noindent}Combing the last claim  with  facts that an operator is compact if and only if its  adjoint operator  is still compact, thus we  get that $K_\alpha: L^p\rightarrow L^1$ is compact if and only if $p<\frac{1}{d+2-\alpha}.$ Then by the following Lemma \ref{cpt}, an  interpolation result of the  compact operators, it implies that  $K_\alpha:L^p\rightarrow L^q ~\text{is~ compact}~\text{if } (\frac{1}{p},\frac{1}{q})\in G(K_\alpha).$
 \qed
 \begin{lem} \cite{CP,Kr}\label{cpt}  Suppose that $1\leq p_1,p_2,q_1,q_2\leq\infty$ and $q_1\neq\infty.$  If a linear operator  $T$ such that $T :L^{p_1}\rightarrow L^{q_1}$ is bounded and $T :L^{p_2}\rightarrow L^{q_2}$ is compact, then $T :L^{p}\rightarrow L^{q}$ is compact, if   there exists  $\theta\in (0,1)$ such that  $$\frac{1}{p}=\frac{\theta}{p_1}+\frac{1-\theta}{p_2}, \frac{1}{q}=\frac{\theta}{q_1}+\frac{1-\theta}{q_2}.$$ 
\end{lem}
\begin{rem} The compactness of $K_\alpha: L^p\rightarrow L^q$ for $1<p,q<\infty$ can be also proved by the Carleson type measure theory on Bergman spaces, see definition for \cite{ZZ,Zhu}. This strategy will be adopted under appropriate circumstances in Section 5.\end{rem}

\section{Proof of Theorem 2 }
In this section we give the proof of Theorem 2. We  first establish  several lemmas.  Denote $k_\alpha(z,w)=\frac{1}{(1-\langle z,w\rangle)^\alpha}, k_\alpha^+(z,w)=\frac{1}{|1-\langle z,w\rangle|^\alpha},z,w\in \mathbb{B}^d.$ Then $k_\alpha,k_\alpha^+$ are integral kernel functions of integral operators $K_\alpha,K_\alpha^+$ respectively.
\begin{lem}\label{ld1} If $0<\alpha\leq d+1,$ then \begin{enumerate}
\item $K_\alpha: L^1\rightarrow L^q$ is bounded if and only if $q<\frac{d+1}{\alpha};$
\item $K_\alpha^+: L^1\rightarrow L^q$ is bounded if and only if $q<\frac{d+1}{\alpha}.$
\end{enumerate}
\end{lem}
\begin{proof} (1) From Proposition 5.2 of \cite{Tao}, we know that  \begin{equation} \Vert K_\alpha\Vert_{ L^1\rightarrow L^q}=\sup_{z\in\mathbb{B}^d} \Vert k_\alpha (z,\cdot)\Vert_{L^q} =\sup_{z\in\mathbb{B}^d}\left (\int \frac{dv(w)}{|1-\langle z,w\rangle|^{q\alpha}}\right)^{\frac{1}{q}}\notag\\
 \end{equation}
Then, combing  with Proposition 1.4.10 of \cite{Rud}, we know that $\Vert K_\alpha\Vert_{ L^1\rightarrow L^q}<\infty$ is equivalent to $q\alpha < d+1.$ It leads the desired result.

(2) It is similar to (1).
\end{proof}
Dually, we have the following lemma.
\begin{lem}\label{kllq} If $0<\alpha\leq d+1,$ then \begin{enumerate}
\item $K_\alpha: L^p\rightarrow L^\infty$ is bounded if and only if $p>\frac{d+1}{(d+1)-\alpha};$
\item $K_\alpha^+: L^p\rightarrow L^\infty$ is bounded if and only if $p>\frac{d+1}{(d+1)-\alpha}.$
\end{enumerate}
\end{lem}
\begin{proof}(1)  From Proposition 5.4 of \cite{Tao}, we know that  \begin{equation}\label{eqk} \Vert K_\alpha\Vert_{ L^p\rightarrow L^\infty}=\sup_{z\in\mathbb{B}^d} \Vert k_\alpha (z,\cdot)\Vert_{L^{p'}} =\sup_{z\in\mathbb{B}^d}\left(\int \frac{dv(w)}{|1-\langle z,w\rangle|^{\frac{p\alpha}{p-1}}}\right)^{\frac{p-1}{p}}.\end{equation} Then, combing  with Proposition 1.4.10 of \cite{Rud}, we know that $\Vert K_\alpha\Vert_{ L^p\rightarrow L^\infty}<\infty$ is equivalent to $\frac{p\alpha}{p-1} < d+1.$ It leads the desired result.

(2) It is similar to (1).
\end{proof}
  \begin{lem}\label{lpq} If $1<p<\frac{d+1}{d+1-\alpha},$ then $K_\alpha: L^p\rightarrow L^q$ is bounded if and only if $\frac{1}{q}\geq \frac{1}{p}+ \frac{\alpha}{d+1}-1.$ 
\end{lem}
  Before proving Lemma \ref{lpq}, we do some preparations. For $p\geq1$, denote Lorentz space $L^{p,\infty}$ on $\mathbb{B}^d$ by 
  $$L^{p,\infty}=\{ f: \sup_{\lambda>0} \lambda d_f^{\frac{1}{p}}(\lambda)<\infty\},$$ where $d_f(\lambda)=v\{z\in\mathbb{B}^d : |f(z)|>\lambda  \}.$ Note that $L^{p,\infty}\subset L^{q,\infty}$ if $p>q,$ and the inclusion is continuous.
  
  \begin{lem}\label{lk} There exists a constant $C$ that only depends on $\alpha$ and $d$ such that,  for every $z\in\mathbb{B}^d,$ $$\Vert k_\alpha(z,\cdot)\Vert_{L^{\frac{d+1}{\alpha},\infty}}=\Vert k_\alpha(\cdot,z)\Vert_{L^{\frac{d+1}{\alpha},\infty}}<C.$$
\end{lem}
\begin{proof} 
By the unitary  invariance of Lebesgue measure, we only need to consider the case of $z=(|z|,0,\cdots,0).$ Note that   \begin{equation}\label{dk}d_{k_\alpha(\cdot,z)}(\lambda)=v\{w\in\mathbb{B}^d : \frac{1}{|1-\langle w, z \rangle|^\alpha} >\lambda \}=v\{w: |\frac{1}{|z|}-w_1|<\frac{1}{|z|}\lambda^{-\frac{1}{\alpha}} \}\end{equation}
When $|z|<\frac{1}{2},$ then $ \frac{1}{|1-\langle w, z \rangle|^\alpha}<2^{\alpha}.$ It follows that $d_{k_\alpha(\cdot,z)}(\lambda)=0, $ if $\lambda\geq 2^{\alpha}.$ Thus $$\Vert k_\alpha(\cdot,z)\Vert_{L^{\frac{d+1}{\alpha},\infty}}\leq 2^{\alpha}.$$  
Now we turn to the case $\frac{1}{2}\leq |z|< 1.$ The conclusion comes immediately  from the following estimation,
\begin{equation}
\lambda d_{k_\alpha(\cdot,z)}^{\frac{\alpha}{d+1}}(\lambda)\leq\begin{cases}\label{D}
1,&\lambda\leq 1,\\
(d\cdot2^{3d-1})^\frac{\alpha}{d+1}, &1<\lambda< \frac{1}{(1-|z|)^\alpha},\\
0,&\lambda\geq \frac{1}{(1-|z|)^\alpha}.
\end {cases}
\end {equation}
Now we prove (\ref{D}). Let $dV(w)=(\frac{i}{2})^d \prod_{n=1}^d dw_n\wedge d\bar{w}_n.$ Then $dV=\frac{\pi^d}{\Gamma(d+1)}dv.$ When $\lambda\leq 1,$ then $\lambda d_{k_\alpha(\cdot,z)}^{\frac{\alpha}{d+1}}(\lambda)<1.$ Denote $I$ by the subset in the unit disk such that 
\begin{equation}\begin{split} I=\{w_1\in \mathbb{D}: |\frac{1}{|z|}-w_1|<\frac{1}{|z|}\lambda^{-\frac{1}{\alpha}} \}. 
 \end{split}\end{equation} 
 When $1<\lambda< \frac{1}{(1-|z|)^\alpha},$ by (\ref{dk}) and Fubini's theorem,  we have that 
\begin{equation}\begin{split}\label{dk1} d_{k_\alpha(\cdot,z)}(\lambda)&=v\{w: |\frac{1}{|z|}-w_1|<\frac{1}{|z|}\lambda^{-\frac{1}{\alpha}}\}\\
   &\leq\frac{\Gamma(d+1)}{\pi^d} (\frac{i}{2})^d \int_{I}dw_1\wedge d\bar{w}_1\int_{|w_2|^2+\cdots+|w_d|^2<1-|w_1|^2} \prod_{n=2}^d dw_n\wedge d\bar{w}_n\\
   &=d \int_{I}(1-|w_1|^2)^{d-1}dv(w_1)\\
   &<d(1-\frac{1}{|z|^2}+2\frac{1}{|z|^2}\frac{1}{\lambda^\frac{1}{\alpha}}-\frac{1}{|z|^2\lambda^\frac{2}{\alpha}})^{d-1}\int_{I}dv(w_1)\\
   &<d\cdot2^{3d-3}\frac{1}{\lambda^\frac{d-1}{\alpha}}\frac{4}{\lambda^\frac{2}{\alpha}}\\
   &=\frac{d\cdot2^{3d-1}}{\lambda^\frac{d+1}{\alpha}}\\
 \end{split}\end{equation}
 Then (\ref{dk1}) implies that $\lambda d_{k_\alpha(\cdot,z)}^{\frac{\alpha}{d+1}}(\lambda)<(d\cdot2^{3d-1})^\frac{\alpha}{d+1}$ if $1<\lambda< \frac{1}{(1-|z|)^\alpha}.$ 
 When $\lambda\geq \frac{1}{(1-|z|)^\alpha},$ it is easy to see that $d_{k_\alpha(\cdot,z)}(\lambda)=0.$ So $\lambda d_{k_\alpha(\cdot,z)}^{\frac{\alpha}{d+1}}(\lambda)=0,$ if  $\lambda\geq \frac{1}{(1-|z|)^\alpha}.$\end{proof}
\begin{cor} \label{} There exists a constant $C$ that only depends on $\alpha$ and $d$ such that,  for every $z\in\mathbb{B}^d,$ $$\Vert k_\alpha^+(z,\cdot)\Vert_{L^{\frac{d+1}{\alpha},\infty}}=\Vert k_\alpha^+(\cdot,z)\Vert_{L^{\frac{d+1}{\alpha},\infty}}<C.$$

\end{cor}
Now we modify Proposition 6.1 of \cite{Tao} to suit our setting.
\begin{lem}\cite{Tao}\label{tao} Suppose that $k: \mathbb{B}^d  \times \mathbb{B}^d \rightarrow \mathbb{C}$ is measurable such that 
$$\Vert k(z,\cdot)\Vert_{L^{r,\infty}}\leq C, z\in \mathbb{B}^d, a.e.$$ and $$\Vert k(\cdot,w)\Vert_{L^{r,\infty}}\leq C, w\in \mathbb{B}^d, a.e.$$ for some $1<r<\infty$ and $C>0.$ Then the operator $T$ defined as 
$$Tf(z)=\int_{\mathbb{B}^d}k(z,w)f(w)dv(w)$$ is bounded   from $L^1$ to $L^{r,\infty}.$ Moreover, if $1<p<q<\infty$ such that $\frac{1}{p}+\frac{1}{r}=\frac{1}{q}+1,$ then $T$ is bounded from $L^p$ to $L^q.$
\end{lem}
\begin{cor}\label{wea} If $0<\alpha\leq d+1,$ then $K_\alpha,K_\alpha^+: L^{1}\rightarrow L^{\frac{d+1}{\alpha},\infty}$ are bounded.
\end{cor}
\begin{proof}
 When $\alpha=d+1,$ $K_{d+1}$ is the  Bergman projection, then $K_{d+1}: L^{1}\rightarrow L^{1,\infty}$ is  bounded by the proof of Theorem 6 of \cite{Mc}. Indeed, similar to the proof of Theorem 6 of \cite{Mc}, by the  Calder\'on-Zygmund decomposition, it can be proved that $K_{d+1}^+: L^{1}\rightarrow L^{1,\infty}$ is  bounded. When $0<\alpha<d+1,$ by Lemma \ref{lk} and Lemma \ref{tao}, it implies that $K_\alpha,K_\alpha^+: L^{1}\rightarrow L^{\frac{d+1}{\alpha},\infty}$ are bounded. It completes the proof.
\end{proof}
{\noindent{\bf{Sufficiency part of Lemma \ref{lpq}.}} We need to prove that  if $1<p<\frac{d+1}{d+1-\alpha}, \frac{1}{q}\geq \frac{1}{p}+ \frac{\alpha}{d+1}-1,$ then $K_\alpha: L^p\rightarrow L^q$ is bounded.
By the continuous embedding of $L$-integrable spaces, it suffices to show that $K_\alpha: L^{p}\rightarrow L^q$ is bounded if $ \frac{1}{q}= \frac{1}{p}+ \frac{\alpha}{d+1}-1.$ Then Lemma \ref{lk} and Lemma \ref{tao} implies that $K_\alpha: L^{p}\rightarrow L^q$ is bounded. \qed
 \begin{cor} \label{KLL} If $1<p<\frac{d+1}{d+1-\alpha},$ then $K_\alpha^+: L^p\rightarrow L^q$ is bounded if $\frac{1}{q}\geq \frac{1}{p}+ \frac{\alpha}{d+1}-1.$ 
 \end{cor}

 For necessity part of Lemma \ref{lpq}, we need find out a function belongs to $L^p$ but its image under $K_\alpha$ is not in $L^q.$  So we establish an isometry from $L^p(\mathbb{D},dv_{d-1})$ to $L^p(\mathbb{B}^d),$ where $\mathbb{D}$ is the unit disc, i.e.  $\mathbb{D}=\mathbb{B}^1.$ Let $$I_p: L^p(\mathbb{D},dv_{d-1}) \rightarrow L^p(\mathbb{B}^d), I_p(f)(z)=f(z_1).$$
 Denote $L_1^p(\mathbb{B}^d)=\{f\in L^p(\mathbb{B}^d): f(z)=f(z_1,0\cdots,0),\forall  z\in \mathbb{B}^d\}.$  If $f\in L_1^p(\mathbb{B}^d)$, we always write $f(z_1)$ without ambiguity.  Denote $A_{0,1}^p$ by the set of holomorphic functions in $L_1^p(\mathbb{B}^d).$  Then we have the following lemma.
 \begin{lem}\label{ip} $I_p$ is an isometry from $A_{d-1}^p(\mathbb{D})$ onto $A_{0,1}^p(\mathbb{B}^d).$ 
 \end{lem}
 \begin{proof}  Suppose that $ f(z_1)\in L_1^p(\mathbb{B}^d),$ then  \begin{equation}\begin{split} \Vert f\Vert_{ L_1^p(\mathbb{B}^d)}^p&=\int_{\mathbb{B}^d}  |f(z_1)|^p dv\\
 &=\frac{\Gamma(d+1)}{\pi^d} (\frac{i}{2})^d \int_{\mathbb{D}}  |f(z_1)|^p dz_1\wedge d\bar{z}_1\int_{|z_2|^2+\cdots+|z_d|^2<1-|z_1|^2} \prod_{n=2}^d dz_n\wedge d\bar{z}_n\\
 &=d\int_{\mathbb{D}}  |f(z_1)|^p (1-|z_1|^2)^{d-1}dv(z_1)\\
 &=\Vert f\Vert_{L^p(\mathbb{D},dv_{d-1})}^p.\\
  \end{split}\end{equation}
  It leads to the desired result.
 \end{proof}
 \begin{cor}\label{la} $A_{0,1}^p(\mathbb{B}^d) \simeq A_{d-1}^p(\mathbb{D}).$
 \end{cor}

   \begin{lem}\label{lap} Suppose that $t\in \mathbb{R},$ then $f(z_1)=\sum_{n=1}^{\infty} n^t z_1^n \in A_{0,1}^p(\mathbb{B}^d) $ if and only if $p(t+1)<d+1.$
  \end{lem}
 \begin{proof} From Corollary 3.5 of \cite{BKV}, Corollary \ref{la} and Proposition 1.4.10 of \cite{Rud}, it yields that $f(z_1)=\sum_{n=1}^{\infty} n^t z_1^n \in A_{0,1}^p(\mathbb{B}^d) \simeq A_{d-1}^p(\mathbb{D})$ if and only if $\sum_{n=1}^{\infty} \frac{\Gamma(n+1+t)}{\Gamma(n+1)\Gamma(t+1)} z_1^n =\frac{1}{(1-z_1)^{t+1}}\in L_a^p(\mathbb{D},dv_{d-1})$ if and only if $p(t+1)<2+(d-1)=d+1.$
 \end{proof}
 
 {\noindent{\bf{Necessity part of Lemma \ref{lpq}.}}
 We need to prove that  if  $1<p<\frac{d+1}{d+1-\alpha}, \frac{1}{q} < \frac{1}{p}+ \frac{\alpha}{d+1}-1,$ then $K_\alpha: L^p\rightarrow L^q$ is unbounded.  Assume that  $K_\alpha: L^p\rightarrow L^q$ is bounded, it is  equivalent to  $K_\alpha: A_0^p\rightarrow A_0^q$ is bounded.  Then $K_\alpha(A_0^p)\subset A_0^q.$ Choose any $t$ such that \begin{equation}\label{tp} \frac{d+1}{q}+d-\alpha <t<\frac{d+1}{p}-1,\end{equation} denote $f_t(z)=\sum_{n=1}^{\infty}n^tz_1^n ,$ then condition (\ref{tp}) and Lemma \ref{lap} imply that $f_t\in A_0^p.$ Now, assume that  $ K_{\alpha}f_t\in A_0^q.$
Then  \begin{equation}\begin{split}\label{bka} K_{\alpha}f_t(z)&=\sum_{n=1}^\infty\frac{\Gamma(n+\alpha)\Gamma(d+1)}{\Gamma(n+d+1)\Gamma(\alpha)}n^tz_1^n\\
&=\frac{\Gamma(d+1)}{\Gamma(\alpha)\Gamma(d+1-\alpha)}\sum_{n=1}^\infty\frac{\Gamma(n+\alpha)\Gamma(d+1-\alpha)}{\Gamma(n+d+1)}n^tz_1^n \\
&=\frac{\Gamma(d+1)}{\Gamma(\alpha)\Gamma(d+1-\alpha)}\sum_{n=1}^\infty B(n+\alpha,d+1-\alpha)n^tz_1^n\\
&\in A_{0,1}^q(\mathbb{B}^d) \simeq A_{d-1}^q(\mathbb{D}),\\
\end{split}\end{equation}
where $B(\cdot,\cdot)$ is Beta function. On the other hand, by Lemma 3.2 of \cite{BKV}, similar to the prove of Lemma 3.4 of \cite{BKV}, it can be proved that  \begin{equation}\label{nbak}\sum_{n=1}^\infty \frac{n^{\alpha-d-1}}{B(n+\alpha,d+1-\alpha)}a_nz^n\in  A_{d-1}^q(\mathbb{D}),\forall \sum_{n=1}^\infty a_nz^n\in A_{d-1}^q(\mathbb{D}).\end{equation}
Combing (\ref{bka}) with (\ref{nbak}), it  implies that $\sum  n^{\alpha-d-1}n^tz_1^n \in A_{d-1}^q(\mathbb{D})\simeq A_{0,1}^q(\mathbb{B}^d).$ So, by Lemma \ref{lap}, it follows that   $q(\alpha-d+t)<d+1,$ namely $t<\frac{d+1}{q}+d-\alpha,$ a contradiction to the condition $\frac{d+1}{q}+d-\alpha <t.$ It completes the proof. \qed

{\noindent{\bf{Proof of Theorem 2.}} 
First, we prove that (1) is equivalent to (3).  Denote  $D_2$  by the region  determined by the equations in (3) of Theorem 3. It is equivalent to prove that $G(K_\alpha)=D_2.$
Lemma \ref{ld1}, Lemma \ref{kllq}, Lemma \ref{lpq} and the convexity of $G(K_\alpha),$ imply that $ G(K_\alpha) =D_2.$ Lemma \ref{ld1}, Lemma \ref{kllq}, Corollary \ref{KLL} and the convexity of $ G(K_\alpha^+) ,$ imply that $D_2\subset G(K_\alpha^+) .$ From the above we know that $G(K_\alpha^+)\subset G(K_\alpha)=D_2.$  Then $G(K_\alpha^+)=D_2.$ It completes the proof.    \qed

\section{Proofs of Theorem 3 and Theorem 4}
In  the previous Section 4, we have characterized completely the $L^p$-$L^q$ boundedness of $K_\alpha,K_\alpha^+$ under the case of $0<\alpha\leq d+1.$  In the present section, we will characterize completely the $L^p$-$L^q$ compactness of $K_\alpha$ when $0<\alpha\leq d+1.$ It is equivalent to solve the  set $F(K_\alpha),$ where $F(K_\alpha)$ is defined by $$F(K_\alpha)=\{(\frac{1}{p},\frac{1}{q})\in E: K_\alpha:L^{p}\rightarrow L^q~ \text{is ~compact}\}.$$ It is easy to see that $F(K_\alpha)$ is a subset of $G(K_\alpha).$ Theorem 3 in fact shows that $F(K_\alpha)$ and $G(K_\alpha)$ differ only by a segment on the boundary of $G(K_\alpha).$ Thus we always show first that  $K_\alpha$ is compact  on the other part of  on the boundary of $G(K_\alpha).$  In the end of this section, we give the proof of Theorem 4. 
\begin{prop}\label{akpqq} $K_{d+1}:L^{p}\rightarrow L^q$ is compact if and only if $1\leq q<p\leq\infty.$
\end{prop}
\begin{proof} From  Theorem 2, we know that $K_{d+1}:L^{p}\rightarrow L^q$ is bounded if and only if $q\leq p.$ Since $K_{d+1}$ is the standard Bergman projection, it is easy to see $K_{d+1}:L^{p}\rightarrow L^p$ is not compact for any $1<p<\infty.$ Thus it suffices to show that $K_\alpha:L^{p}\rightarrow L^q$ is compact if $q<p.$ Indeed, this can be proved  by the similar method we used in the $step$ 2 of proof of  Theorem 1, we omit it.
\end{proof}

Now, we  recall some results on hypergeometric function theory for later use. For complex numbers $\alpha,\beta,\gamma$ and complex variable $z,$ we use the classical notation $\tensor[_2]{F} {_1}(\alpha,\beta;\gamma;z) $  to denote
$$\tensor[_2]{F} {_1}(\alpha,\beta;\gamma;z) =\sum_{j=0}^{\infty}\frac{(\alpha)_j(\beta)_j}{j!(\gamma)_j}z^j,$$ with $\gamma\neq 0,-1,-2,\ldots,$  where $(\alpha)_j=\Pi_{k=0}^{j-1}(\alpha+k)$ is the Pochhammer for any complex number $\alpha.$ The following lemma is in fact a restatement of Proposition 1.4.10 of \cite{Rud}.
\begin{lem}\cite{Rud}\label{hi1} Suppose $\beta\in\mathbb{R}$ and $\gamma>-1,$ then
$$ \int_{\mathbb{B}^d}\frac{(1-|w|^2)^\gamma}{|1-\langle z,w\rangle|^{2\beta}}dv(w)=\frac{\Gamma(1+d)\Gamma(1+\gamma)}{\Gamma(1+d+\gamma)}\tensor[_2]{F} {_1}(\beta,\beta;1+d+\gamma;|z|^2).$$
\end{lem}
 
 We also need the following lemma.
\begin{lem}\cite[Chapter 2]{EMO} \label{hi2} The following three identities hold. \begin{enumerate}
\item $\tensor[_2]{F} {_1}(\alpha,\beta;\gamma;z) =(1-z)^{\gamma-\alpha-\beta}  \tensor[_2]{F} {_1}(\gamma-\alpha,\gamma-\beta;\gamma;z);$
\item $\tensor[_2]{F} {_1}(\alpha,\beta;\gamma;1)=\frac{\Gamma(\gamma)\Gamma(\gamma-\alpha-\beta)}{\Gamma(\gamma-\alpha)\Gamma(\gamma-\beta)},$ if $Re(\gamma-\alpha-\beta)>0;$
\item $\frac{d}{dz} ~\tensor[_2]{F} {_1}(\alpha,\beta;\gamma;z)=\frac{\alpha\beta}{\gamma}~ \tensor[_2]{F} {_1}(\alpha+1,\beta+1;\gamma+1;z).$
\end{enumerate}\end{lem}

\begin{lem} \label{kcom} If $0<\alpha<d+1,$  then $K_\alpha:L^\infty \rightarrow L^q$ is compact for any  $ 1\leq q\leq\infty.$
\end{lem}
\begin{proof} Since the continuous embedding of $L$-integrable spaces, it suffices to prove that $K_\alpha:L^\infty \rightarrow L^\infty$ is compact. We first prove that, for any $f\in L^\infty,$ then $K_\alpha f \in A(\mathbb{B}^d),$ where $A(\mathbb{B}^d)=H(\mathbb{B}^d)\cap C(\overline{\mathbb{B}^d})$ is the ball algebra. For $f\in L^\infty,$ it is clear that $K_\alpha f$ is holomorphic on the ball, i.e. $K_\alpha f \in H(\mathbb{B}^d).$ Now we prove that $K_\alpha f$ is also continuous on the closed ball $\overline{\mathbb{B}^d.}$ From Lemma \ref{hi1} and (2) of Lemma \ref{hi2}, it implies that $K_\alpha f(\eta)$ exists  for any $\eta\in \partial \mathbb{B}^d$ and $$\vert K_\alpha f(\eta)\vert \leq\Vert f\Vert_\infty \frac{\Gamma(d+1)\Gamma(d+1-\alpha)}{\Gamma^2(d+1-\frac{\alpha}{2})}.$$ We now turn to prove that $K_\alpha f$ is continuous on $\partial \mathbb{B}^d.$ It suffices to prove that, for any $\eta\in \partial \mathbb{B}^d$ and for any point sequence $\{ z_n\}$ in $\mathbb{B}^d $ satisfying that  $z_n\rightarrow \eta,$ we have $K_\alpha f(z_n)\rightarrow K_\alpha f(\eta)$ as $n\rightarrow \infty.$ By Lemma \ref{hi1} and (2) of Lemma \ref{hi2} again, we have 

 \begin{equation}\begin{split}\label{kiie}\vert K_\alpha f(z)\vert &\leq \Vert f\Vert_\infty \int_{\mathbb{B}^d}\frac{1}{\vert 1-\langle z,w\rangle\vert^\alpha}dv(w)\\
 &\leq\Vert f\Vert_\infty \int_{\mathbb{B}^d}\frac{1}{\vert 1-\langle \eta,w\rangle\vert^\alpha}dv(w)\\&= \Vert f\Vert_\infty \frac{\Gamma(d+1)\Gamma(d+1-\alpha)}{\Gamma^2(d+1-\frac{\alpha}{2})},\end{split}  \end{equation}  for any $z\in \mathbb{B}^d.$ Due to the absolute continuity of the integral, it implies that, for any $\varepsilon>0,$ there exists $0<\delta<1$, satisfying that \begin{equation}\label{iif}\int_{F}\frac{dv(w)}{\vert 1-\langle \eta,w\rangle\vert^\alpha}\leq \frac{\varepsilon}{4},\end{equation}  whenever $v(F)<\delta.$ Denote $F_\delta=\{z\in \mathbb{B}^d: \sqrt[d]{1-\frac{\delta}{2}}<\vert z\vert<1\}.$ Note that $v(F_\delta)=\frac{\delta}{2}<\delta$  and $$\frac{1}{(1-\langle z_n,w\rangle)^\alpha}\rightarrow \frac{1}{(1-\langle \eta ,w\rangle)^\alpha}~\text{uniformly~on }~\mathbb{B}^d \setminus F_\delta,\text{~as}~n\rightarrow \infty.$$ 
Then there exists $N>0$ such that, for any $n>N,$ $$\int_{\mathbb{B}^d\setminus F_\delta}\left\vert\frac{1}{( 1-\langle z_n,w\rangle)^\alpha}-\frac{1}{( 1-\langle \eta,w\rangle)^\alpha}\right\vert dv(w)\leq \frac{\varepsilon}{2}.$$
Combing this with (\ref{kiie}), (\ref{iif}), it implies that, for any $n>N,$ 
 \begin{equation}\begin{split}\label{keps} \vert K_\alpha f(z_n)-K_\alpha f(\eta)\vert&\leq \Vert f\Vert_\infty\int_{\mathbb{B}^d\setminus F_\delta}\left\vert\frac{1}{(1-\langle z_n,w\rangle)^\alpha}-\frac{1}{(1-\langle \eta,w\rangle)^\alpha}\right\vert dv(w) \\
&~~~~~~~~~+\Vert f\Vert_\infty\int_{F_\delta}\left\vert\frac{1}{(1-\langle z_n,w\rangle)^\alpha}-\frac{1}{( 1-\langle \eta,w\rangle)^\alpha}\right\vert dv(w)\\
&\leq \Vert f\Vert_\infty\int_{\mathbb{B}^d\setminus F_\delta}\left\vert\frac{1}{(1-\langle z_n,w\rangle)^\alpha}-\frac{1}{( 1-\langle \eta,w\rangle)^\alpha}\right\vert dv(w) \\
&~~~~~~~~~+2\Vert f\Vert_\infty\int_{F_\delta}\frac{1}{\vert 1-\langle \eta,w\rangle\vert^\alpha} dv(w)\\
&\leq  \Vert f\Vert_\infty\frac{\varepsilon}{2}+2 \Vert f\Vert_\infty\frac{\varepsilon}{4}\\
&= \varepsilon \Vert f\Vert_\infty.\\
\end{split}  \end{equation} 
It completes the proof of what $K_\alpha f$ is continuous on the closed ball $\overline{\mathbb{B}^d}.$ Now we prove that, for any bounded sequence  in $L^\infty,$ there exists a subsequence satisfying that its image under $K_\alpha $ is convergent in $L^\infty.$ Suppose that $\{f_n\}$ is a bounded sequence in $L^\infty,$ then we have $\{K_\alpha f_n\}$ is in $C(\overline{\mathbb{B}^d})$ and $\{K_\alpha f_n\}$ is uniformly bounded by (\ref{kiie}).  Now we prove that $\{K_\alpha f_n\}$ is also equicontinuous. 
 From (\ref{keps}), we know that \begin{equation}\label{limv}\lim_{\mathbb{B}^d\ni z\rightarrow\eta} \int_{\mathbb{B}^d}\left\vert\frac{1}{(1-\langle z,w\rangle)^\alpha}-\frac{1}{( 1-\langle \eta,w\rangle)^\alpha}\right\vert dv(w)=0,\end{equation} for arbitrary fixed $\eta\in\partial\mathbb{B}^d.$ Combing (\ref{limv}) with the unitary invariance of Lebsgue measure and the symmetry of the unit ball, it implies that,  for any $\epsilon>0,$ there exists $0<\delta'<1,$ satisfying that  \begin{equation}\label{limee}\int_{\mathbb{B}^d}\left\vert\frac{1}{(1-\langle z,w\rangle)^\alpha}-\frac{1}{( 1-\langle \eta,w\rangle)^\alpha}\right\vert dv(w)\leq\frac{\epsilon}{2}\end{equation} whenever $z\in \mathbb{B}^d,\eta\in\partial\mathbb{B}^d$ and $\vert z-\eta\vert<\delta'.$ Denote $B_{1-\frac{\delta'}{2}}=\{z\in \mathbb{C}^d: \vert z\vert \leq1-\frac{\delta'}{2}\}$ and $C_{\frac{\delta'}{2}}=\{z\in \mathbb{C}^d: 1-\frac{\delta'}{2}<\vert z\vert \leq1\}.$ Then the closed ball $\overline{\mathbb{B}^d}$ has the following decomposition,  \begin{equation}\label{dec}\overline{\mathbb{B}^d}=B_{1-\frac{\delta'}{2}}\cup C_{\frac{\delta'}{2}} \text{~and~} B_{1-\frac{\delta'}{2}}\cap C_{\frac{\delta'}{2}} =\emptyset.\end{equation} Since the function  $\frac{1}{(1-\langle z,w\rangle)^\alpha}$  is uniformly continuous on compact set $B_{1-\frac{\delta'}{2}}\times\overline{\mathbb{B}^d},$ then there exists $0<\delta''<1$ such that \begin{equation}\label{limw}\left\vert\frac{1}{(1-\langle z_1,w\rangle)^\alpha}-\frac{1}{(1-\langle z_2,w\rangle)^\alpha}\right\vert\leq\epsilon,\end{equation}     whenever $(z_1,w),(z_2,w)\in B_{1-\frac{\delta'}{2}}\times\overline{\mathbb{B}^d}$ and $\vert z_1-z_2\vert<\delta''.$ Take $\delta'''=\min\{\frac{\delta'}{2},\delta''\}.$ Now we prove that, for any $z_1,z_2\in \overline{\mathbb{B}^d}$ such that $\vert z_1-z_2\vert<\delta''',$ then we have 
 \begin{equation}\label{limepp}\int_{\mathbb{B}^d}\left\vert\frac{1}{(1-\langle z_1,w\rangle)^\alpha}-\frac{1}{( 1-\langle z_2,w\rangle)^\alpha}\right\vert dv(w)\leq\epsilon.\end{equation} 
In fact, there are two cases need to be considered.  The first case is $z_1\in C_{\frac{\delta'}{2}}$ or  $z_2\in C_{\frac{\delta'}{2}}.$ Without loss of generality, we can assume that $z_1\in C_{\frac{\delta'}{2}},$ then there exists 
an $\eta \in \partial \mathbb{B}^d$ satisfying that $\vert z_1-\eta\vert<\delta'''\leq\frac{\delta'}{2}.$ By triangle inequality, it implies that $\vert z_2-\eta\vert\leq\vert z_2-z_1\vert +\vert z_1-\eta\vert<\delta'.$ Together with (\ref{limee}), it implies that \begin{equation}\begin{split}\label{limeef}&\int_{\mathbb{B}^d}\left\vert\frac{1}{(1-\langle z_1,w\rangle)^\alpha}-\frac{1}{( 1-\langle z_2,w\rangle)^\alpha}\right\vert dv(w)\notag\\
&\leq\int_{\mathbb{B}^d}\left\vert\frac{1}{(1-\langle z_1,w\rangle)^\alpha}-\frac{1}{( 1-\langle \eta,w\rangle)^\alpha}\right\vert dv(w)\\&~\vspace{0.2cm}+\int_{\mathbb{B}^d}\left\vert\frac{1}{(1-\langle \eta, w\rangle)^\alpha}-\frac{1}{( 1-\langle z_2,w\rangle)^\alpha}\right\vert dv(w)\\&\leq \epsilon
\end{split}\end{equation} The second case is $z_1,z_2\in B_{1-\frac{\delta'}{2}}.$ By (\ref{limw}), it implies that \begin{equation}\begin{split}\label{limeeff}\int_{\mathbb{B}^d}\left\vert\frac{1}{(1-\langle z_1,w\rangle)^\alpha}-\frac{1}{( 1-\langle z_2,w\rangle)^\alpha}\right\vert dv(w)\leq \epsilon  \int_{\mathbb{B}^d}dv= \epsilon.\notag\end{split}\end{equation}
It proves (\ref{limepp}). Combing with $$\vert K_\alpha f_n(z_1)-K_\alpha f_n(z_2)\vert\leq\Vert f_n\Vert_\infty\int_{\mathbb{B}^d}\left\vert\frac{1}{(1-\langle z_1,w\rangle)^\alpha}-\frac{1}{( 1-\langle z_2,w\rangle)^\alpha}\right\vert dv(w),$$ it implies that $\{K_\alpha f_n\}$ is equicontinuous. Then by Arzel\`a-Ascoli theorem, it implies that $\{K_\alpha f_n\}$  has a convergency subsequence in the supremum norm.
\end{proof}
\begin{cor}\label{kkcom} If $0<\alpha<d+1,$ then the following holds:
\begin{enumerate}
\item $K_\alpha: L^p\rightarrow L^1$ is compact for any $1\leq p\leq \infty.$
\item $K_\alpha: L^1\rightarrow L^q$ is compact if and only if $q<\frac{d+1}{\alpha}.$
\item $K_\alpha: L^p\rightarrow L^\infty$ is compact if and only if $p>\frac{d+1}{d+1-\alpha}.$
\end{enumerate}
\begin{proof} It comes from Lemma \ref{cpt}, Lemma \ref{kcom} and the fact that $K_\alpha $ is adjoint.
\end{proof}
\end{cor}
In the following, we deal with the  case $1< p, q<\infty$ . However, we need the following result about  Carleson type measures for the Bergman spaces over the unit ball.
\begin{lem}\cite{ZZ}\label{carl} Suppose $1\leq p\leq q<\infty$ and $\mu$ is a positive Borel measure on $\mathbb{B}^d$. Then the following conditions are equivalent:
\begin{enumerate}
\item If $\{f_n\}$ is a bounded sequence in $A_{0}^p$ and $f_n(z) \rightarrow 0$ for every $z \in \mathbb{B}^d$, then
$$\lim_{n\rightarrow\infty}\int_{\mathbb{B}^d}\vert f_n\vert^qd\mu=0.$$
\item For every (or some) $s>0,$ we have $$\lim_{\vert z\vert\rightarrow1^-}\int_{\mathbb{B}^d}\frac{(1-\vert z\vert^2)^s}{\vert 1-\langle z,w\rangle\vert^{s+\frac{q(d+1)}{p}}}d\mu(w)=0.$$
\end{enumerate}
\end{lem}
 The Borel measure in Lemma \ref{carl} is in fact  the so-called vanishing Carleson  measures. If denote $A^q(d\mu) $ by the weighted Bergman space $A^q(d\mu)=H(\mathbb{B}^d)\cap L^q(\mathbb{B}^d,d\mu).$ Then (1) of Lemma \ref{carl} guarantees ( or is equivalent to)  that  the embedding $Id: A^p_{0}\rightarrow A^q(d\mu)$ is compact.
  \begin{prop}\label{Acpt} If  $0<\alpha<d+1$ and $1< p\leq q<\infty,$  then the following conditions are equivalent.
\begin{enumerate}
\item $K_\alpha: L^p\rightarrow L^q$ is compact.
\item $K_\alpha: A_0^p\rightarrow A_0^q$ is compact.
\item The embedding $Id: A_0^p\rightarrow A_{q(d+1-\alpha)}^q$ is compact.
\item $\frac{1}{q}> \frac{1}{p}+ \frac{\alpha}{d+1}-1.$
\end{enumerate}
 \end{prop}
\begin{proof} We first prove that (1) is equivalent to (2). Clearly (1) implies  (2). To prove the reverse, note that by $0<\alpha<d+1$ and combing this with Theorem 2 gives us that $K_{d+1}:L^p\rightarrow A_0^p$ is bounded.
Suppose that $\{f_n\}$ is  an arbitrary bounded sequence in $L^p,$ thus we have $\{K_{d+1}f_n\}$ is a bounded sequence in $A_0^p.$ Then the compactness of operator  $K_\alpha: A_0^p\rightarrow A_0^q$ implies that, there exists a subsequence $\{f_{n_j}\}$  such that $\{K_\alpha( K_{d+1}f_{n_j})\}$ is convergent in $A_0^q.$ Combing this with Lemma \ref{3k}, yields that $\{K_\alpha f_{n_j}\}$ is convergent in $A_0^q.$ It proved that (2) implies (1).

{\noindent}Now we prove that (2) is equivalent  to (3). Similar to the proof of  Proposition \ref{kka}, by Theorem 14 of \cite{ZZ} and Theorem 2.19 of \cite{Zhu}, it can be proved that $$R^{\alpha-d-1,d+1-\alpha}:A_0^q \rightarrow A_{q(d+1-\alpha)}^q$$ and its inverse operator are bounded. Note that $K_\alpha=R^{0,\alpha-d-1}$   on  $A_0^p$ and $$R^{\alpha-d-1,d+1-\alpha}R^{0,\alpha-d-1}f=f,\forall f\in A_0^p.$$
 Then we have the following decomposition for embedding $Id,$ \begin{equation}\begin{split}\label{ida} Id:A_0^p&\xrightarrow{K_\alpha=R^{0,\alpha-d-1}} A_0^q \xrightarrow{R^{\alpha-d-1,d+1-\alpha}}A_{q(d+1-\alpha)}^q,\\&Id= R^{\alpha-d-1,d+1-\alpha}K_\alpha .\end{split} \end{equation}
 Combing (\ref{ida}) with the fact that $R^{\alpha-d-1,d+1-\alpha}$ has bounded inverse, it implies that $K_\alpha: A_0^p\rightarrow A_0^q$ is compact if and only if the embedding $Id: A_0^p\rightarrow A_{q(d+1-\alpha)}^q$ is compact.
 
 {\noindent}In the next, we prove that (3) is equivalent to (4). Suppose that $\{ f_n\}$ is an  arbitrary bounded sequence in $A_0^p,$ then by Theorem 20 of \cite{ZZ}, the locally estimation for functions in $A_0^p,$ it implies that 
 $\{ f_n\}$ is a normal family. Hence, by Fatou's lemma, similar to the proof of Theorem 1, there exists a subsequence $\{ f_{n_j}\}$ and $g\in A_0^p$ such that $f_{n_j}$ converges uniformly to $g$ on any compact subset of $\mathbb{B}^d.$ Then $\{ f_{n_j}-g\}$ is in $ A_0^p$ and $f_{n_j}-g\rightarrow 0$ pointwise as $j\rightarrow \infty.$ Together with Lemma \ref{carl}, it yields that the embedding $Id: A_0^p\rightarrow A_{q(d+1-\alpha)}^q$ is compact if and only if  \begin{equation}\label{l0}\lim_{\vert z\vert\rightarrow1^-}\int_{\mathbb{B}^d}\frac{(1-\vert z\vert^2)^s}{\vert 1-\langle z,w\rangle\vert^{s+\frac{q(d+1)}{p}}}dv_{q(d+1-\alpha)}(w)=0,\end{equation} for any $s>0.$ On the other hand, by Proposition 1.4.10 of \cite{Rud}, it implies that (\ref{l0}) is equivalent to $\frac{1}{q}> \frac{1}{p}+ \frac{\alpha}{d+1}-1.$ It completes the proof. \end{proof}
{\noindent{\bf{Proof of Theorem 3.}}  When $\alpha=d+1.$ Theorem 3 degenerates into Proposition \ref{akpqq}.

 {\noindent}Now we turn to the case $0<\alpha<d+1.$ We first prove that (2) implies (1). In fact, it is an immediate  corollary from Lemma \ref{cpt}, Lemma \ref{kcom} and Corollary \ref{kkcom}. To see the reverse, note that by Theorem 2  and combining this with Proposition \ref{Acpt} gives that (1) is not held if (2) is not held, it implies that  (1) implies (2), completing  the proof. \qed
 
 {\noindent{\bf{Proof of Theorem 4.}} By Theorem 1,2,3, it is easy to see that $(1)\Rightarrow (4) \Rightarrow (2) \Leftrightarrow(3).$ Thus  we only need to show that $(2)\Rightarrow (1).$ It is equivalent to prove that $(2)$ is not held if $(1)$ is not held. It suffices to show that $K_\alpha: L^\infty\rightarrow L^1$ is not bounded if $\alpha\geq d+2.$ Suppose that $\alpha\geq d+2.$ In view to Theorem 7.1 of \cite{Zhu}, it implies that $K_\alpha (L^\infty)=\mathcal{B}_{\alpha-d}.$ From Lemma \ref{ba}, we know that $\mathcal{B}_{\alpha-d}\not\subset L^1,$ it means that $K_\alpha: L^\infty\rightarrow L^1$ is not bounded. It completes the proof. \qed

\section{norm estimations for  $K_\alpha.$}

In the previous sections, we have completely characterized  the $L^p$-$L^q$ boundedness of $K_\alpha,K_\alpha^+$ and compactness of $K_\alpha.$ In the present section, we will  state and prove  some sharp norm estimates of $K_\alpha,K_\alpha^+,$ which  gives essentially the upper bounds of the best constants in the  Hardy-Littlewood-Sobolev inequalities. 
\begin{prop}\label{ktion}If $d+1<\alpha<d+2$ and $K_\alpha: L^p\rightarrow L^q$ is bounded, then
\begin{equation}\label{nom} \Vert K_\alpha\Vert_{ L^p\rightarrow L^q}\leq\frac{\Gamma(d+1)^{1+\frac{1}{q}-\frac{1}{p}}\Gamma(\alpha-(d+1))\Gamma(\frac{1}{q^{-1}-p^{-1}}(d+1-\alpha)+1)^{\frac{1}{q}-\frac{1}{p}}}{\Gamma(\frac{\alpha}{2})^2\Gamma(\frac{1}{q^{-1}-p^{-1}}(d+1-\alpha)+d+1)^{\frac{1}{q}-\frac{1}{p}}}.\end{equation}\end{prop}

\begin{lem} Suppose that $d
+1<\alpha<d+2$ and $(0,\frac{1}{q})\in G(K_\alpha)=G(K_\alpha^+),$ then the following holds. \begin{enumerate}
\item $ \Vert K_\alpha\Vert_{L^{\infty}\rightarrow L^q}\leq\Vert K_\alpha^+\Vert_{L^{\infty}\rightarrow L^q}=\Vert \int_{\mathbb{B}^d} k_\alpha^+(\cdot,w)dv(w)\Vert_{L^q}.$ 
\item In particular, when $d=1,$ \begin{equation}\label{eme} \Vert K_\alpha\Vert_{L^{\infty}\rightarrow L^1}\leq\Vert K_\alpha^+\Vert_{L^{\infty}\rightarrow L^1}=\frac{4}{(\alpha-2)^2}\left(\frac{\Gamma(3-\alpha)}{\Gamma^2(2-\frac{\alpha}{2})}-1\right).\end{equation}
\item For any general $(0,\frac{1}{q})\in G(K_\alpha)=G(K_\alpha^+),$ \begin{small} \begin{equation}\label{nm} \Vert K_\alpha\Vert_{L^{\infty}\rightarrow L^q}\leq\Vert K_\alpha^+\Vert_{L^{\infty}\rightarrow L^q}\leq \frac{\Gamma(d+1)^{1+\frac{1}{q}}\Gamma(\alpha-(d+1))\Gamma(q(d+1-\alpha)+1)^{\frac{1}{q}}}{\Gamma(\frac{\alpha}{2})^2\Gamma(q(d+1-\alpha)+d+1)^\frac{1}{q}}. \end{equation}\end{small}
\end{enumerate}
\end{lem}
\begin{proof} (1) Since $|K_\alpha(f)|\leq K_\alpha^+(|f|),$ it  implies   that $\Vert K_\alpha\Vert_{L^{\infty}\rightarrow L^q}\leq\Vert K_\alpha^+\Vert_{L^{\infty}\rightarrow L^q}$ if $K_\alpha$ and $K_\alpha^+ $ are bounded. Note that  $\vert K_\alpha^+f\vert (z)\leq \Vert f\Vert_\infty \int_{\mathbb{B}^d}\frac{1}{\vert 1-\langle z,w\rangle\vert^{\alpha}}dv(w) $ for any $f\in L^\infty,$ it yields  that $$\Vert K_\alpha\Vert_{L^{\infty}\rightarrow L^q}\leq \Vert K_\alpha^+\Vert_{L^{\infty}\rightarrow L^q}\leq\Vert \int_{\mathbb{B}^d} k_\alpha^+(\cdot,w)dv(w)\Vert_{L^q}.$$ 
To see the reverse, note  that  $$\Vert K_\alpha^+\Vert_{L^{\infty}\rightarrow L^q}\geq \Vert K_\alpha^+1\Vert_{ L^q}=\Vert \int_{\mathbb{B}^d} k_\alpha^+(\cdot,w)dv(w)\Vert_{L^q}.$$ It leads to the desired result.
 
{\noindent}(2) Now we turn  to calculate the norm in the case of $d=1.$ 
 From  (2) of Lemma \ref{hi2} and what we have proved, it follows that 
\begin{equation}\begin{split}
 \Vert K_\alpha\Vert_{L^{\infty}\rightarrow L^1}&\leq\Vert K_\alpha^+\Vert_{L^{\infty}\rightarrow L^1}\\
 &= \int_{\mathbb{B}^d}\tensor[_2]{F} {_1}(\frac{\alpha}{2},\frac{\alpha}{2};d+1;\vert z\vert^2)dv(z)\\
 &=d\int_0^1\tensor[_2]{F} {_1}(\frac{\alpha}{2},\frac{\alpha}{2};d+1;r)r^{d-1}dr,\\
\end{split}\end{equation}
in the last equality we apply the integration in polar coordinates, see Lemma 1.8 of \cite{Zhu}, and the unitary invariance of hypergeometric function $\tensor[_2]{F} {_1}(\frac{\alpha}{2},\frac{\alpha}{2};d+1;\vert z\vert^2).$ Now we use the differential properties listed in Lemma \ref{hi2} to calculate the integral in the case of $d=1.$ We observe (3) of Lemma \ref{hi2}, it gives that 
\begin{equation}
\frac{d}{dr}\left( \tensor[_2]{F} {_1}(\frac{\alpha}{2}-1,\frac{\alpha}{2}-1;1;r)\right)=(\frac{\alpha}{2} -1)^2\tensor[_2]{F} {_1}(\frac{\alpha}{2},\frac{\alpha}{2};2;r).\notag\end{equation}
Then integrate the two sides of the above equality and we get
\begin{equation}\int_0^1\tensor[_2]{F} {_1}(\frac{\alpha}{2},\frac{\alpha}{2};2;r)dr=\frac{4}{(\alpha-2)^2}\left( \tensor[_2]{F} {_1}(\frac{\alpha}{2}-1,\frac{\alpha}{2}-1;1;1)-1\right).\notag\end{equation}
Together with with (2) of Lemma \ref{hi2} yields the desired result. 

{\noindent}(3) Combing   (1) with Lemma \ref{hi1} and (1),(2) of Lemma \ref{hi2}, it follows that
\begin{equation}\begin{split} \label{nf}\Vert K_\alpha^+\Vert_{L^{\infty}\rightarrow L^q}&=\left( \int_{\mathbb{B}^d}\left(\int_{\mathbb{B}^d}\frac{1}{\vert 1-\langle z,w\rangle\vert^{\alpha}}dv(w)\right)^{q}dv(z) \right)^\frac{1}{q}\notag\\
&= \left( \int_{\mathbb{B}^d} \tensor[_2]{F} {_1}(\frac{\alpha}{2},\frac{\alpha}{2};d+1;\vert z\vert^2)^qdv(z)\right)^\frac{1}{q}\\
&=\left(\int_{\mathbb{B}^d} (1-\vert z \vert^2)^{q(d+1-\alpha)}\tensor[_2]{F} {_1}(d+1-\frac{\alpha}{2},d+1-\frac{\alpha}{2};d+1;\vert z\vert^2)^qdv(z)\right)^\frac{1}{q}\\
&\leq \tensor[_2]{F} {_1}(d+1-\frac{\alpha}{2},d+1-\frac{\alpha}{2};d+1;1)\left(\int_{\mathbb{B}^d} (1-\vert z \vert^2)^{q(d+1-\alpha)}dv(z)\right)^\frac{1}{q}\\
&=   \frac{\Gamma(d+1)^{1+\frac{1}{q}}\Gamma(\alpha-(d+1))\Gamma(q(d+1-\alpha)+1)^{\frac{1}{q}}}{\Gamma(\frac{\alpha}{2})^2\Gamma(q(d+1-\alpha)+d+1)^\frac{1}{q}}.\\
\end{split}\end{equation}
It leads to (\ref{nm}).
\end{proof}

{\noindent{\bf{Proof of Proposition \ref{ktion}.}} Suppose that  $K_\alpha^+:L^p\rightarrow L^q$ is bounded, it is equivalent to $(\frac{1}{p},\frac{1}{q})\in G(K_\alpha^+).$ Then (3) of Theorem 1 guarantees 
$\frac{1}{q}-\frac{1}{p}>\alpha-(d+1).$ Note by (3) of Theorem 1 again yields that 
\begin{equation}\label{oqp} (0, \frac{1}{q}-\frac{1}{p}),(1-(\frac{1}{q}-\frac{1}{p}),1)\in G(K_\alpha^+)\end{equation} and there exists $0\leq \theta\leq 1$ satisfying that \begin{equation}\label{thet}(\frac{1}{p},\frac{1}{q})=\theta\cdot(0, \frac{1}{q}-\frac{1}{p})+(1-\theta)\cdot(1-(\frac{1}{q}-\frac{1}{p}),1).\end{equation}
Combing (\ref{oqp}),(\ref{thet}) with Lemma \ref{int}, it follows that 

\begin{equation}\label{nmpq}\Vert K_\alpha^+\Vert_{L^p\rightarrow L^q} \leq \Vert K_\alpha^+\Vert_{L^\infty\rightarrow L^{\frac{1}{q^{-1}-p^{-1}}}}^\theta \Vert K_\alpha^+\Vert_{L^{\frac{1}{1-(q^{-1}-p^{-1})}}\rightarrow L^1}^{1-\theta}
\end{equation}
We observe that   the adjoint  operator of $K_\alpha^+:L^\infty\rightarrow L^{\frac{1}{q^{-1}-p^{-1}}}$ is exactly the operator $K_\alpha^+:L^{\frac{1}{1-(q^{-1}-p^{-1})}}\rightarrow L^1,$ it means that $$ \Vert K_\alpha^+\Vert_{L^\infty\rightarrow L^{\frac{1}{q^{-1}-p^{-1}}}}= \Vert K_\alpha^+\Vert_{L^{\frac{1}{1-(q^{-1}-p^{-1})}}\rightarrow L^1}.$$
Applying this to (\ref{nmpq}), then yields that  \begin{equation}\label{nmpq1} \Vert K_\alpha^+\Vert_{L^p\rightarrow L^q} \leq \Vert K_\alpha^+\Vert_{L^\infty\rightarrow L^{\frac{1}{q^{-1}-p^{-1}}}}.\end{equation}
Combing (\ref{nmpq1}) with (\ref{oqp}) and applying Lemma \ref{nm}, it leads to the desired conclusion.\qed
\begin{cor} Suppose that $C_1$ is the best constant in HLS 1, then 
$$C_1\leq \frac{\Gamma(d+1)^{2-\frac{1}{s}-\frac{1}{p}}\Gamma(\alpha-(d+1))\Gamma(\frac{1}{1-s^{-1}-p^{-1}}(d+1-\alpha)+1)^{1-\frac{1}{s}-\frac{1}{p}}}{\Gamma(\frac{\alpha}{2})^2\Gamma(\frac{1}{1-s^{-1}-p^{-1}}(d+1-\alpha)+d+1)^{1-\frac{1}{s}-\frac{1}{p}}}.$$
\end{cor}
\begin{prop}\label{alno}   If $0<\alpha<d+1$ and  $\frac{1}{p}-(1- \frac{\alpha}{d+1})<\frac{1}{q}\leq\frac{1}{p},$ then \begin{equation}\label{iek}\Vert K_\alpha\Vert_{L^p\rightarrow L^q} \leq\Vert K_\alpha^+\Vert_{L^p\rightarrow L^q} \leq \left (\frac{\Gamma(d+1)\Gamma(d+1-\frac{\alpha}{1-(p^{-1}-q^{-1})})}{\Gamma^2(d+1-\frac{\alpha}{2(1-(p^{-1}-q^{-1}))})}\right)^{1-(\frac{1}{p}-\frac{1}{q})}.\end{equation} In particular, when $q=\infty,$ the inequality (\ref{iek}) is an equality.
\end{prop}
\begin{proof} We first prove that  (\ref{iek}) is in fact equality under the case of $q=\infty.$ From (\ref{eqk}), we know that  \begin{equation}\label{eqkk}\Vert K_\alpha\Vert_{ L^p\rightarrow L^\infty} = \Vert K_\alpha^+\Vert_{ L^p\rightarrow L^\infty} =\sup_{z\in\mathbb{B}^d}\left(\int \frac{dv(w)}{|1-\langle z,w\rangle|^{\frac{p\alpha}{p-1}}}\right)^{\frac{p-1}{p}}.\end{equation} 
On the other hand, Lemma \ref{hi1} and (2) of Lemma \ref{hi2} yield 
 \begin{equation}\begin{split}\label{est} \int_{\mathbb{B}^d} \frac{dv(w)}{|1-\langle z,w\rangle|^{\frac{p\alpha}{p-1}}}&= \tensor[_2]{F} {_1}(\frac{p\alpha}{2(p-1)},\frac{p\alpha}{2(p-1)};d+1;\vert z\vert^2)\\  
 &\leq \tensor[_2]{F} {_1}(\frac{p\alpha}{2(p-1)},\frac{p\alpha}{2(p-1)};d+1;1)\\
 &=\frac{\Gamma(d+1)\Gamma(d+1-\frac{p\alpha}{p-1})}{\Gamma^2(d+1-\frac{p\alpha}{2(p-1)})}.\end{split}  \end{equation}  Combing (\ref{eqkk}) and (\ref{est}), it implies that \begin{equation}\label{iiek}\Vert K_\alpha\Vert_{L^p\rightarrow L^\infty} =\Vert K_\alpha^+\Vert_{L^p\rightarrow L^\infty} = \left (\frac{\Gamma(d+1)\Gamma(d+1-\frac{p\alpha}{p-1})}{\Gamma^2(d+1-\frac{p\alpha}{2(p-1)})}\right)^{\frac{p-1}{p}}.\end{equation}
Now we turn to prove (\ref{iek}) in the general case. Note first that  $|K_\alpha(f)|\leq K_\alpha^+(|f|),$ it  implies   that  $\Vert K_\alpha\Vert_{L^p\rightarrow L^q}\leq\Vert K_\alpha^+\Vert_{L^p\rightarrow L^q}$ if $K_\alpha$ and $K_\alpha^+ $ are bounded. Since   $\frac{1}{p}-(1- \frac{\alpha}{d+1})<\frac{1}{q}\leq\frac{1}{p},$ Theorem 2 implies that \begin{equation}\label{pqta}(\frac{1}{p},\frac{1}{q}),(\frac{1}{p}-\frac{1}{q},0),(1,1-(\frac{1}{p}-\frac{1}{q}))\in G(K_\alpha^+)\end{equation} and there exists $0\leq\theta\leq1$ satisfying that \begin{equation}\label{pqtt} (\frac{1}{p},\frac{1}{q})=\theta\cdot(\frac{1}{p}-\frac{1}{q},0)+(1-\theta)\cdot(1,1-(\frac{1}{p}-\frac{1}{q}))\end{equation} Combing (\ref{pqta}),(\ref{pqtt}) with Lemma \ref{int}, it follows that 
\begin{equation}\label{knor}\Vert K_\alpha^+\Vert_{L^p\rightarrow L^q} \leq \Vert K_\alpha^+\Vert_{L^{\frac{1}{p^{-1}-q^{-1}}}\rightarrow L^\infty}^\theta \Vert K_\alpha^+\Vert_{L^1\rightarrow L^{\frac{1}{1-(p^{-1}-q^{-1})}}}^{1-\theta}
\end{equation}
We observe that   the adjoint  operator of $K_\alpha^+:L^{\frac{1}{p^{-1}-q^{-1}}}\rightarrow L^\infty$ is exactly the operator $K_\alpha^+:L^1\rightarrow L^{\frac{1}{1-(p^{-1}-q^{-1})}},$ it means that \begin{equation}\label{keq}\Vert K_\alpha^+\Vert_{L^{\frac{1}{p^{-1}-q^{-1}}}\rightarrow L^\infty}= \Vert K_\alpha^+\Vert_{L^1\rightarrow L^{\frac{1}{1-(p^{-1}-q^{-1})}}}.\end{equation} Thus by (\ref{knor}) and (\ref{keq}), it follows that $$\Vert K_\alpha^+\Vert_{L^p\rightarrow L^q} \leq \Vert K_\alpha^+\Vert_{L^{\frac{1}{p^{-1}-q^{-1}}}\rightarrow L^\infty}.$$
Together with (\ref{iiek}), it completes the proof.\end{proof}
\begin{cor} Suppose that $C_2$ is the best constant in HLS 2, then the following holds.
\begin{enumerate}
\item If $\frac{1}{p}<1-\frac{1}{s},$ then $$C_2\leq\frac{\Gamma(d+1)\Gamma(d+1-\alpha)}{\Gamma^2(d+1-\frac{\alpha}{2})}.$$
\item If  $\frac{1}{p}-(1- \frac{\alpha}{d+1})<1-\frac{1}{s}\leq\frac{1}{p},$ then $$ C_2 \leq \left (\frac{\Gamma(d+1)\Gamma(d+1-\frac{\alpha}{2-p^{-1}-s^{-1}})}{\Gamma^2(d+1-\frac{\alpha}{2(2-p^{-1}-s^{-1})})}\right)^{2-(\frac{1}{p}-\frac{1}{s})}.$$
\end{enumerate}
\end{cor}

{\noindent{\bf{Proof of Theorem 5.}} When $\alpha<\frac{d+2}{2},$ by Proposition 1.4.10 of \cite{Rud}, it implies that the kernel function $k_{\alpha}^+\in L^2(\mathbb{B}^d\times\mathbb{B}^d,dv\times dv),$ thus $K_\alpha,K_\alpha^+:L^2\rightarrow L^2$ are Hilbert-Schmidt. Note that \begin{equation}\label{tra}Tr(K_\alpha^*K_\alpha)=\int_{\mathbb{B}^d}\int_{\mathbb{B}^d}\frac{1}{\vert 1-\langle z, w\rangle\vert^{2\alpha}}dv(w)dv(z).\end{equation} When $\alpha\neq1,$ similar to  (\ref{eme}), yields the trace formula. Now we deal with the spacial case $\alpha=1.$ Combing Lemma \ref{hi1} with (\ref{tra}), it implies that  $$Tr(K_1^*K_1)=\int_0^1\tensor[_2]{F} {_1}(1,1;2;r)dr =\sum_{j=1}^\infty \frac{1}{j^2}=\frac{\pi^2}{6}.$$ \qed
\begin{rem} By (3) of Proposition \ref{hi2} and inductive method, we can get explicit trace formulas for every dimension $d\geq 1.$
\end{rem}
 As a  consequence of Theorem 5 we obtain the following generalized Euler-Jacobi identity.
 \begin{cor} Suppose that $0<\alpha<\frac{3}{2},$ then \begin{equation}\label{Eul}\sum_{j=0}^\infty \left(\frac{\Gamma(\alpha+j)}{\Gamma(\alpha)\Gamma(2+j)}\right)^2=\frac{1}{(\alpha-1)^2}\left (\frac{\Gamma(3-2\alpha)}{\Gamma^2(2-\alpha)} -1\right ).\end{equation}
 \end{cor}
When $\alpha=1,$ the identity (\ref{Eul}) is the well known Euler-Jacobi identity $$\sum_{j=1}^\infty \frac{1}{j^2}=\frac{\pi^2}{6}.$$
When $d=1, 0<\alpha<\frac{3}{2},$ we know that $K_\alpha: L^2\rightarrow L^2$ is compact by Theorem 1 or  Theorem 5. Thus the spectrum $\sigma(K_\alpha)$ of the operator $K_\alpha$ is exactly the point spectrum. Note that every $K_\alpha$ is adjoint, then combing (\ref{KR}) and (\ref{Eul}) with Stirling's formula, we have the following.
 \begin{cor} Suppose that $d=1$ and $0<\alpha<\frac{3}{2},$ then $K_\alpha: L^2\rightarrow L^2$ is compact and   $$\sigma(K_\alpha)=\bigcup_{j=0}^\infty\{\frac{\Gamma(\alpha+j)}{\Gamma(\alpha)\Gamma(2+j)}\}.$$ Moreover, in this case, $$\Vert K_\alpha\Vert_{L^2\rightarrow L^2}=\max_{0\leq j\leq \infty} \frac{\Gamma(\alpha+j)}{\Gamma(\alpha)\Gamma(2+j)}.$$\end{cor}




{\noindent{\bf{Acknowledgements.}}  
The first author would  like to thank Professor G. Zhang for his helpful discussions and  warm hospitality  when the author visited Chalmers University of  Technology.

\bibliographystyle{plain}
 
\end{document}